\documentclass[11pt]{article}

\usepackage{setspace}

\usepackage{amsmath,amssymb,amsthm}
\usepackage{epsfig}
\usepackage[latin1]{inputenc}
\usepackage{psfrag}
\usepackage{graphicx} 
\usepackage{epstopdf} 

\usepackage{tikz}
\usetikzlibrary{calc}

\begin{document}

\newtheorem{theorem}{Theorem}
\newtheorem{proposition}[theorem]{Proposition}
\newtheorem{remark}[theorem]{Remark}
\newtheorem{lemma}[theorem]{Lemma}
\newcommand{\RR}{\mathbb{R}}
\newcommand{\NN}{\mathbb{N}}
\newcommand{\ZZ}{\mathbb{Z}}
\newcommand{\PP}{\mathbb{P}}
\newcommand{\EE}{\mathbb{E}}
\newcommand{\Var}{\mathbb{V}{\rm ar}}
\newcommand{\Cov}{\mathbb{C}{\rm ov}}
\newcommand{\II}{{\cal{I}}}
\newcommand{\Id}{\text{Id}}
\newcommand{\dive}{\operatorname{div}}
\newcommand{\eps}{\varepsilon}

\newcommand{\dps}{\displaystyle}
\newcommand{\dis}{\displaystyle}

\def\longrightharpoonup{\relbar\joinrel\rightharpoonup}

\title{A control variate approach based on a defect-type theory for variance reduction in stochastic homogenization}
\author{F. Legoll$^{1,2}$ and W. Minvielle$^{3,2}$\\
{\footnotesize $^1$ Laboratoire Navier, \'Ecole Nationale des Ponts et
Chauss\'ees, Universit\'e Paris-Est,}\\
{\footnotesize 6 et 8 avenue Blaise Pascal, 77455 Marne-La-Vall\'ee
Cedex 2, France}\\
{\footnotesize \tt legoll@lami.enpc.fr}\\
{\footnotesize $^2$ INRIA Rocquencourt, MATHERIALS research-team,}\\
{\footnotesize Domaine de Voluceau, B.P. 105,
78153 Le Chesnay Cedex, France}\\
{\footnotesize $^3$ CERMICS, \'Ecole Nationale des Ponts et
Chauss\'ees, Universit\'e Paris-Est,}\\
{\footnotesize 6 et 8 avenue Blaise Pascal, 77455 Marne-La-Vall\'ee
Cedex 2, France}\\
{\footnotesize \tt william.minvielle@cermics.enpc.fr}\\
}
\date{\today}

\maketitle

\begin{abstract}
We consider a variance reduction approach for the stochastic homogenization of divergence form linear elliptic problems. Although the exact homogenized coefficients are deterministic, their practical approximations are random. We introduce a control variate technique to reduce the variance of the computed approximations of the homogenized coefficients. Our approach is based on a surrogate model inspired by a defect-type theory, where a perfect periodic material is perturbed by rare defects. This model has been introduced in~\cite{ALB_cras} in the context of weakly random models. In this work, we address the fully random case, and show that the perturbative approaches proposed in~\cite{ALB_cras,ALB_mms} can be turned into an efficient control variable.

We theoretically demonstrate the efficiency of our approach in simple cases. We next provide illustrating numerical results and compare our approach with other variance reduction strategies. We also show how to use the Reduced Basis approach proposed in~\cite{clb_thomines} so that the cost of building the surrogate model remains limited. 
\end{abstract}

\section{Introduction}

In this work, we introduce a variance reduction approach based on the control variate technique for the homogenization of the following stochastic, elliptic, linear problem:
\begin{equation}
\label{eq:pb0-stoc}
-\dive\left(A\left(\frac{x}{\varepsilon},
\omega\right) \nabla u^\varepsilon \right) = f \ \ \text{in ${\mathcal D}$}, 
\qquad
u^\varepsilon(\cdot,\omega) = 0 \ \ \text{on $\partial {\mathcal D}$},
\end{equation}
set on a bounded domain $\mathcal D$ in $\RR^d$, where $f$ is a deterministic function in $L^2(\mathcal D)$.
The random matrix $A$ is assumed to be uniformly elliptic, bounded and stationary in a sense made precise below. 

It is well-known that, in the limit when $\varepsilon$ goes to 0, the above problem converges to the homogenized problem
\begin{equation}
\label{eq:pb0-stoc-homog}
-\dive\left(A^\star\nabla u^\star \right) = f \ \ \text{in ${\mathcal D}$}, 
\qquad
u^\star = 0 \ \ \text{on $\partial {\mathcal D}$},
\end{equation}
where the homogenized matrix $A^\star$ is deterministic, and given by an expectation of an integral involving the so-called corrector function, that solves a random auxiliary problem set on the {\em entire} space. In practice, the corrector problem is approximated by a problem set on a {\em bounded} domain $Q_N$ (see Section~\ref{sec:bp} below for details). A by-product of this truncation procedure is that the {\em deterministic} matrix $A^\star$ is in practice approximated by a {\em random}, apparent homogenized matrix $A^\star_N(\omega)$. Randomness therefore comes again into the picture. In this work, we introduce a variance reduction approach to obtain practical approximations of $A^\star$ with a smaller variance. Our approach is a control variate technique, which is based on a surrogate random model, simple enough to allow for easier computations, and close enough to the reference model to eventually improve the accuracy. 

We mention that, in our previous works~\cite{mprf,banff,cedya}, we have already proposed variance reduction approaches to compute better approximations of $A^\star$. We used there the technique of antithetic variables, which is a generic variance reduction approach. In addition, we have shown in~\cite{LM13} that this technique carries over to nonlinear stochastic homogenization problems, when the problem at hand is formulated as a variational convex problem. In this work, we return to the linear equation~\eqref{eq:pb0-stoc}, and design an approach based on the control variate technique, where a surrogate model is used to improve the computational efficiency. Our approach here is therefore much more specific to the problem at hand than the antithetic variable approaches proposed previously. We therefore expect this technique to provide better results. This is indeed the case, as discussed along the numerical examples of Section~\ref{sec:num_low}.

Generally speaking, control variate approaches are based on using surrogate models as a kind of preconditioner (see Section~\ref{sec:CV_101} below for more details). In this work, the surrogate model that we use is inspired by a defect-type model, introduced in~\cite{ALB_cras,ALB_ccp,ALB_mms} in the context of weakly random models. The model considered there is that of a perfect periodic material perturbed by rare defects. These defects may introduce a significant change in the local properties of the random matrix $A(x,\omega)$. However they only occur with a small probability $\eta$. In that setting, when $\eta$ is small, the authors of~\cite{ALB_cras,ALB_ccp,ALB_mms} have shown that a good approximation of the homogenized properties can be obtained by only solving deterministic problems rather than random problems, as usually required in stochastic homogenization. In this work, we build our surrogate model upon the ideas of~\cite{ALB_cras,ALB_ccp,ALB_mms}. However, we address the regime when $\eta$ is not small, hence perturbative approaches are not accurate enough.

\medskip

Our article is organized as follows. In the sequel of this introduction, we present in more details some basic elements of stochastic homogenization, situate the questions under consideration in a more general setting, and introduce the control variate approach in a general setting (see Section~\ref{sec:CV_101}). In Section~\ref{sec:weakly_random}, we recall the weakly stochastic model introduced in~\cite{ALB_cras,ALB_ccp,ALB_mms}. 

Next, in Section~\ref{sec:CV}, we describe how to use this weakly stochastic model to build surrogate models that can be used in the ``fully random'' (non perturbative) regime. We introduce two control variate approaches. The first approach (see Section~\ref{sec:order-1}) is based on a first-order weakly stochastic approach, where defects are considered as {\em isolated} from one another. The second one (see Section~\ref{sec:order-2}) is based on a second-order weakly stochastic approach, where {\em pairs of defects} are considered. The main qualitative difference between these two control variate approaches is that the second one takes into account the geometry, whereas the first one essentially only depends on $\dis \int_{Q_N} A(x,\omega) \, dx$. It is well known that, in dimension $d \geq 2$, geometry -- i.e. the way different materials are located one with respect to the other -- matters in the homogenization process. The fact that our second approach takes into account the geometry is thus a very interesting feature.  

We next collect in Section~\ref{sec:theo} some elements of theoretical analysis. We first consider the one-dimensional case (Section~\ref{sec:theo_1D}) and provide there a complete analysis of our approach (see Propositions~\ref{prop:theo_1D} and~\ref{prop:theo_1D_suite}). We show that the variance of the apparent homogenized coefficient scales as $N^{-1}$ (where $N$ is the size of the large domain on which, in practice, the corrector problem is solved), while it is decreased to $N^{-2}$ (resp. $N^{-3}$) when using our first-order (resp. second-order) control variate approach. In Section~\ref{sec:theo_dD}, we next turn to the multi-dimensional case. Our main result is Lemma~\ref{lem:variance-scaling}. 

Section~\ref{sec:num} is devoted to numerical experiments. We quantitatively demonstrate the efficiency of our approach on two test cases in Sections~\ref{sec:num_low} and~\ref{sec:num_high}. As pointed out above, our second approach is based on considering pairs of defects. In order to keep limited the offline cost associated to building the surrogate model, we show in Section~\ref{sec:num_RB} that it is possible to use the Reduced Basis approach introduced in~\cite{clb_thomines}: the precomputation cost is then dramatically decreased, while the gain in variance with respect to a Monte Carlo approach remains similar. 

\subsection{Homogenization theoretical setting}

To begin with, we introduce the basic setting of stochastic homogenization we employ. We refer to~\cite{papa} for some seminal contribution, to~\cite{engquist-souganidis} for a general, numerically oriented presentation, and to~\cite{blp,cd,jikov} for classical textbooks. We also refer to~\cite{enumath} and the review article~\cite{singapour} (and the extensive bibliography contained therein) for a  presentation of our particular setting. Throughout this article, $(\Omega, {\mathcal F}, \PP)$ is a probability space and we denote by $\dps \EE(X) = \int_\Omega X(\omega) d\PP(\omega)$ the expectation of any random variable $X\in L^1(\Omega, d\PP)$. We next fix $d\in {\mathbb N}^\star$ (the ambient physical dimension), and assume that the group $(\ZZ^d, +)$ acts on $\Omega$. We denote by $(\tau_k)_{k\in \ZZ^d}$ this action, and assume that it preserves the measure $\PP$, that is, for all $\dps k\in \ZZ^d$ and all $A \in {\cal F}$, $\dps \PP(\tau_k A) = \PP(A)$. We assume that the action $\tau$ is {\em ergodic}, that is, if $A \in {\mathcal F}$ is such that $\tau_k A = A$ for any $k \in \ZZ^d$, then $\PP(A) = 0$ or 1. In addition, we define the following notion of stationarity (see~\cite{cras-stoc,jmpa}): a function $F \in L^1_{\rm loc}\left(\RR^d, L^1(\Omega)\right)$ is {\em stationary} if
\begin{equation}
\label{eq:stationnarite-disc}
\forall k \in \ZZ^d, \quad
F(x+k, \omega) = F(x,\tau_k\omega)
\quad
\text{a.e. in $x$ and a.s.}
\end{equation}
In this setting, the ergodic theorem~\cite{krengel,shiryaev,tempelman} can be stated as follows: 
{\it Let $F\in L^\infty\left(\RR^d, L^1(\Omega)\right)$ be a stationary random
variable in the above sense. For $k = (k_1,k_2, \dots, k_d) \in \ZZ^d$,
we set $\displaystyle |k|_\infty = \sup_{1\leq i \leq d} |k_i|$. Then
$$
\frac{1}{(2N+1)^d} \sum_{|k|_\infty\leq N} F(x,\tau_k\omega)
\mathop{\longrightarrow}_{N\rightarrow \infty}
\EE\left(F(x,\cdot)\right) \quad \mbox{in }L^\infty(\RR^d),
\mbox{ almost surely}.
$$
This implies (denoting by $Q$ the unit cube in $\RR^d$) that
$$
F\left(\frac x \varepsilon ,\omega \right)
\mathop{\longrightharpoonup}_{\varepsilon \rightarrow 0}^*
\EE\left(\int_Q F(x,\cdot)dx\right) \quad \mbox{in }L^\infty(\RR^d),
\mbox{ almost surely}.
$$
}

Besides technicalities, the purpose of the above setting is simply to formalize that, even though realizations may vary, the function $F$ at point $x \in \RR^d$ and the function $F$ at point $x+k$, $k \in \ZZ^d$, share the same law. In the homogenization context we now turn to, this means that the local, microscopic environment (encoded in the matrix field $A$ in~\eqref{eq:pb0-stoc}) is everywhere the same \emph{on average}. From this, homogenized, macroscopic properties will follow. In addition, and this is evident reading the above setting, the microscopic environment has a relation to an underlying \emph{periodic} structure (thus the integer shifts $k$ in (\ref{eq:stationnarite-disc})).  

\medskip

We consider problem (\ref{eq:pb0-stoc}), where ${\mathcal D}$ is an open, 
bounded domain of $\RR^d$ and where $f \in L^2({\mathcal D})$ is deterministic. The random matrix $A$ is assumed stationary in the sense of~\eqref{eq:stationnarite-disc}. We also assume that $A$ is bounded and that, in the sense of quadratic forms, $A$ is positive and almost surely bounded away from zero: there exist deterministic constants $c$ and $C$ such that, almost surely,
\begin{equation}
\label{eq:elliptic}
\| A (\cdot,\omega) \|_{L^\infty(\RR^d)} \leq C
\quad \text{and} \quad
\forall \xi \in \RR^d, \quad \xi^T A(x,\omega) \xi \geq c \xi^T \xi
\quad \text{a.e.}
\end{equation}
In this specific setting, the solution $u^\eps(\cdot,\omega)$ to~\eqref{eq:pb0-stoc} converges (when $\eps$ goes to 0) to the solution $u^\star$ to the homogenized problem~\eqref{eq:pb0-stoc-homog} almost surely, weakly in $H^1({\cal D})$ and strongly in $L^2({\cal D})$. The homogenized matrix $A^\star$ that appears in (\ref{eq:pb0-stoc-homog}) reads
\begin{equation}
\label{eq:matrice-homogeneisee}
\forall p \in \RR^d, \quad
A^\star \, p = {\mathbb E} \left[ \int_Q A(x,\cdot) \ (\nabla w_p(x,\cdot) + p) \, dx \right],
\qquad Q=(0,1)^d,
\end{equation}
where, for any vector $p \in \RR^d$, the \emph{corrector} $w_p$ is the solution (unique up to the addition of a random constant) to the following corrector problem:
\begin{equation}
\label{eq:correcteur-complet}
\left\{
\begin{array}{l}
\dps -\dive\left[A(\nabla w_p + p)\right] = 0 \quad
\text{in } \RR^d \ \text{ a.s.},
\vspace{6pt}\\
\dps \nabla w_p \text{ is stationary in the sense of~\eqref{eq:stationnarite-disc},} \qquad
\int_Q {\mathbb E}(\nabla w_p) = 0.
\end{array}\right.
\end{equation}

\subsection{Practical approximation of the homogenized matrix}
\label{sec:bp}

The corrector problem~\eqref{eq:correcteur-complet} is set on the entire space $\RR^d$, and is therefore challenging to solve. Approximations are in order. In practice, the deterministic matrix $A^\star$ is approximated by the random matrix $A^\star_N (\omega)$ defined by
\begin{equation}
\label{eq:A-N}
\forall p \in \RR^d, \quad A^\star_N (\omega) \, p = 
\frac{1}{|Q_N|} \int_{Q_N} A(x,\omega) \left(p + \nabla w_p^N(x,\omega)
\right) \, dx,
\end{equation}
which is obtained by solving the corrector problem on a \emph{truncated} domain, say the cube $Q_N = (-N/2,N/2)^d$:
\begin{equation} 
\label{eq:correcteur-random-N}
-\dive \left(A(\cdot,\omega)
\left( p +  \nabla w_p^N(\cdot,\omega)\right) \right) = 0,
\qquad
w_p^N(\cdot,\omega) \ \mbox{is $Q_N$-periodic}.
\end{equation}
As briefly explained above, although $A^\star$ itself is a deterministic object, its practical approximation $A^\star_N$ is random. It is only in the limit of infinitely large domains $Q_N$ that the deterministic value is attained. Indeed, as shown in~\cite{bourgeat}, we have 
$$
\lim_{N \to \infty} A^\star_N(\omega) = A^\star \quad \text{almost surely.} 
$$

Many studies have been recently devoted to establishing sharp estimates on the convergence of the random apparent homogenized quantities (computed on $Q_N$) to the exact deterministic homogenized quantities. We refer e.g. to~\cite{bourgeat,GNO_invent,nolen,yuri} and to the comprehensive discussion of~\cite[Section 1.2]{mprf}. We take here the problem from a slightly different perspective. We observe that the error
$$
A^\star - A_N^\star(\omega)
=
\Big(
A^\star - \EE \left[ A_N^\star \right]
\Big)
+
\Big(
\EE \left[ A_N^\star \right]
- A_N^\star(\omega)
\Big)
$$
is the sum of a systematic error and of a statistical error (the first and second terms in the above right-hand side, respectively). We focus here on the {\em statistical error}, and propose approaches to reduce the confidence interval of empirical means approximating $\EE \left[ A_N^\star \right]$, for a given truncated domain $Q_N$. Optimal estimates on the variance of $A_N^\star$ have been established in~\cite[Theorem 1.3 and Proposition 1.4]{nolen}. For a setting slightly different from ours (namely for homogenization problems set on random {\em lattices}), optimal estimates on the systematic and statistical errors have been established in~\cite[Theorem 2]{GNO_invent}. The authors noted there that ``the systematic error is much smaller than the statistical error'', in the sense that the latter decays with a slower rate with respect to $N$ than the former. For large values of $N$, the statistical error (that we address in this work) is therefore dominating over the systematic error.

\medskip

A standard technique to compute an approximation of $\EE \left[ \left( A^\star_N \right)_{ij} \right]$ (for any entry $ij$) is to consider $M$ independent and identically distributed realizations of the field $A$, solve for each of them the corrector problem~\eqref{eq:correcteur-random-N} (thereby obtaining i.i.d. realizations $A^{\star,m}_N(\omega)$) and proceed following a Monte Carlo approach:
\begin{equation}
\label{eq:MC_estim_homog}
\EE \left[ \left( A^\star_N \right)_{ij}\right]
\approx
I^{\rm MC}_M := \frac{1}{M} \sum_{m=1}^M \left( A^{\star,m}_N(\omega) \right)_{ij}.
\end{equation}
In view of the Central Limit Theorem, we know that our quantity of interest $\EE \left[ \left( A^\star_N \right)_{ij} \right]$ asymptotically lies in the confidence interval
$$
\left[
I^{\rm MC}_M - 1.96 \frac{\sqrt{\Var \left[ \left( A^\star_N \right)_{ij} \right]}}{\sqrt{M}}
,
I^{\rm MC}_M + 1.96 \frac{\sqrt{\Var \left[ \left( A^\star_N \right)_{ij} \right]}}{\sqrt{M}}
\right]
$$
with a probability equal to 95 \%. 

In this article, we show that, using a control variate approach, we can design a practical approach that, for any finite $N$, allows to compute a better approximation of $\EE \left[ \left( A^\star_N \right)_{ij} \right]$ 
than $I^{\rm MC}_M$. Otherwise stated, for an equal computational cost, we obtain a more accurate (i.e. with a smaller confidence interval) approximation. 

\subsection{Control variate approach}
\label{sec:CV_101}

Before presenting our specific approach, we describe here the control variate approach in a general context (see~\cite[page 277]{fishman}). Consider a general probability space $(\Omega, {\cal F}, \PP)$ and a scalar random variable $X \in L^2(\Omega, \RR)$. Our aim is to compute its expectation $\EE(X)$. In the sequel, we will use that approach for the random variable $\left( A^\star_N(\omega) \right)_{ij}$, for any entry $1 \leq i, j \leq d$.

As always, a first possibility is to resort to $M$ i.i.d. realizations of $X$, denoted $X^m(\omega)$ for $1 \leq m \leq M$. The expectation is then approximated by the Monte Carlo empirical mean 
$$
I^{\rm MC}_M := \frac{1}{M} \sum_{m=1}^M X^m(\omega)
$$
and we know that, with a probability equal to 95 \%, $\EE \left[ X \right]$ asymptotically lies in the confidence interval
\begin{equation}
\label{eq:confidence_MC}
\left[
I^{\rm MC}_M - 1.96 \frac{\sqrt{\Var \left[ X \right]}}{\sqrt{M}}
,
I^{\rm MC}_M + 1.96 \frac{\sqrt{\Var \left[ X \right]}}{\sqrt{M}}
\right].
\end{equation}
To reduce the variance of the estimation, consider now a random variable $Y \in L^2(\Omega, \RR)$, the expectation of which is analytically known. Then, for any scalar deterministic parameter $\rho$ to be fixed later, we consider the \emph{controlled variable}
\begin{equation}
\label{eq:control-variate-begin}
D_\rho(\omega) = X(\omega) - \rho \Big( Y(\omega) - \EE[Y] \Big).
\end{equation}
Since $\EE[Y]$ is known exactly, sampling realizations of $D_\rho$ amounts to sampling realizations of $X$ and $Y$. We obviously have $\EE[D_\rho] = \EE[X]$. To approximate $\EE[X]$, the control variate approach consists in performing a standard Monte Carlo approximation on $D_\rho$. We hence consider $M$ i.i.d. realizations of $D_\rho$, denoted $D_\rho^m(\omega)$, introduce the empirical mean 
$$
I^{\rm CV}_M := \frac{1}{M} \sum_{m=1}^M D_\rho^m(\omega)
$$
and write that, with a probability equal to 95 \%, $\EE[D_\rho] = \EE \left[ X \right]$ asymptotically lies in the confidence interval
\begin{equation}
\label{eq:confidence_CV}
\left[
I^{\rm CV}_M - 1.96 \frac{\sqrt{\Var \left[ D_\rho \right]}}{\sqrt{M}}
,
I^{\rm CV}_M + 1.96 \frac{\sqrt{\Var \left[ D_\rho \right]}}{\sqrt{M}}
\right].
\end{equation}
If $\rho$ and $Y$ are such that $\Var \left[ D_\rho \right] < \Var \left[ X \right]$, then the width of the above confidence interval is smaller than that of~\eqref{eq:confidence_MC}, and hence we have built a more accurate approximation of $\EE \left[ X \right]$.

\medskip

We now detail how to choose $\rho$ and $Y$ in~\eqref{eq:control-variate-begin}. Suppose for now that $Y$ is given. We wish to pick $\rho$ such that the variance of $D_\rho$ is minimal. Writing that 
$$
\Var[D_\rho] = \Var[X]- 2\rho \Cov [X,Y] + \rho^2 \Var[Y],
$$
we see that the optimal value of $\rho$ reads
\begin{equation}
\label{eq:rho_star}
\rho^\star 
= 
\text{argmin} \ \Var[D_\rho]
=
\frac {\Cov[X,Y]}{\Var[Y]}.
\end{equation}
For this choice, we have, using the Cauchy-Schwarz inequality,
$$
\Var[D_{\rho^\star}] 
= 
\Var[X] \left( 1 - \frac{\left( \Cov[X,Y] \right)^2}{\Var[X] \Var[Y]} \right) \leq \Var[X].
$$
We thus observe that, for any choice of $Y$, we can choose $\rho$ such that the variance of $D_\rho$ is indeed smaller than that of $X$. Of course, the ratio of variances $\dis \frac{\Var[D_{\rho^\star}]}{\Var[X]}$, which is directly related to the gain in accuracy, depends on $Y$, and more precisely on the value of $\dps \frac{\left( \Cov[X,Y] \right)^2}{\Var[X] \Var[Y]}$. The larger the correlation between $X$ and $Y$, the better. In contrast to the choice of $\rho$, the choice of $Y$ is problem dependent. In addition, the control variable $Y$ needs to be {\em random}.

\begin{remark}
\label{rem:choix_rho}
In practice, we do not have access to the optimal value~\eqref{eq:rho_star}, which involves exact expectations. One possibility (which is the one we adopt in this work) is to replace~\eqref{eq:rho_star} by the empirical estimator
$$
\rho^\star 
\approx 
\frac{
\sum_{m=1}^M (X^m(\omega) - \mu_M(X)) \, (Y^m(\omega) - \EE[Y])
}{
\sum_{m=1}^M (Y^m(\omega) - \EE[Y])^2
},
$$
where $\dis \mu_M(X) = \frac{1}{M} \sum_{m=1}^M X^m(\omega)$. This choice corresponds to minimizing with respect to $\rho$ the empirical variance of $D_\rho$ defined as $\dis \frac{1}{M} \sum_{m=1}^M \left( D^m_\rho(\omega) - \mu_M(X) \right)^2$, where $D^m_\rho(\omega) = X^m(\omega) - \rho \Big( Y^m(\omega) - \EE[Y] \Big)$.
\end{remark}

\section{A weakly random setting: rare defects in a periodic structure}
\label{sec:weakly_random}

As pointed out above, the surrogate model that we use to build our controlled variable is inspired by a defect-type model, introduced in~\cite{ALB_cras,ALB_ccp,ALB_mms} in the context of weakly random models, and that we describe now. 

\subsection{Presentation of the model}

Assume that, in~\eqref{eq:pb0-stoc}, the random matrix $A$ is of the form
\begin{equation}
\label{eq:weak-random}
A(x,\omega) = A_\eta(x,\omega) = A_{\rm per}(x) + b_\eta(x,\omega) \Big( C_{\rm per}(x) - A_{\rm per}(x) \Big)
\end{equation}
where $A_{\rm per}$ and $C_{\rm per}$ are $\ZZ^d$-periodic matrices that are bounded and positive in the sense of~\eqref{eq:elliptic}, and
\begin{equation}
\label{eq:random-form1}
b_\eta(x,\omega) = \sum_{k \in \mathbb{Z}^d}  \mathbf{1}_{Q + k}(x) B^\eta_k(\omega),
\end{equation}
where $(B^\eta_k)_{k \in \ZZ^d}$ are i.i.d. scalar random variables. The matrix $A$ is indeed stationary in the sense of~\eqref{eq:stationnarite-disc}. We furthermore assume that $B_k^\eta$ follows a Bernoulli law of parameter $\eta \in (0,1)$:
\begin{equation}
\label{eq:ber}
\PP(B^\eta_k = 1) = \eta, \quad \PP(B^\eta_k = 0) = 1-\eta.
\end{equation}
The matrix $A(x,\omega)$ then satisfies assumption~\eqref{eq:elliptic}. 

In each cell $Q+k$, the field $A$ is equal to $A_{\rm per}$ with the probability $1-\eta$, and equal to $C_{\rm per}$ with the probability $\eta$. When $\eta$ is small, then~\eqref{eq:weak-random}--\eqref{eq:random-form1}--\eqref{eq:ber} models a periodic material (described by $A_{\rm per}$) that is randomly perturbed (and then described by $C_{\rm per}$). The perturbation is rare when $\eta$ is small (therefore the material is described by $A_{\rm per}$ ``most of the time''), and thus it can be considered as a defect. However, the perturbation is not small in $L^\infty$ norm: $\left\| C_{\rm per} - A_{\rm per} \right\|_{L^\infty}$ is not assumed to be small. We refer to~\cite{ALB_mms} for practical examples motivating this framework.

On Fig.~\ref{fig:cercles}, we show two realizations of the field $A_\eta(x,\omega)$ (on the domain $Q_N$ for $N=20$) for some specific choices of $A_{\rm per}$ and $C_{\rm per}$ (see~\cite[Fig. 4.2]{ALB_mms} for more details). On the right part of that figure, we set $\eta = 0.4$, which is close to the value $\eta=1/2$, when defects are as frequent as non-defects.

Note that specifying $A_\eta(x,\omega)$ on $Q_N$ simply amounts to specifying the values of $B_k^\eta(\omega)$ for all $k$ such that $k+Q \subset Q_N$.

\begin{figure}[htbp]
\begin{center}
\includegraphics[width=6cm]{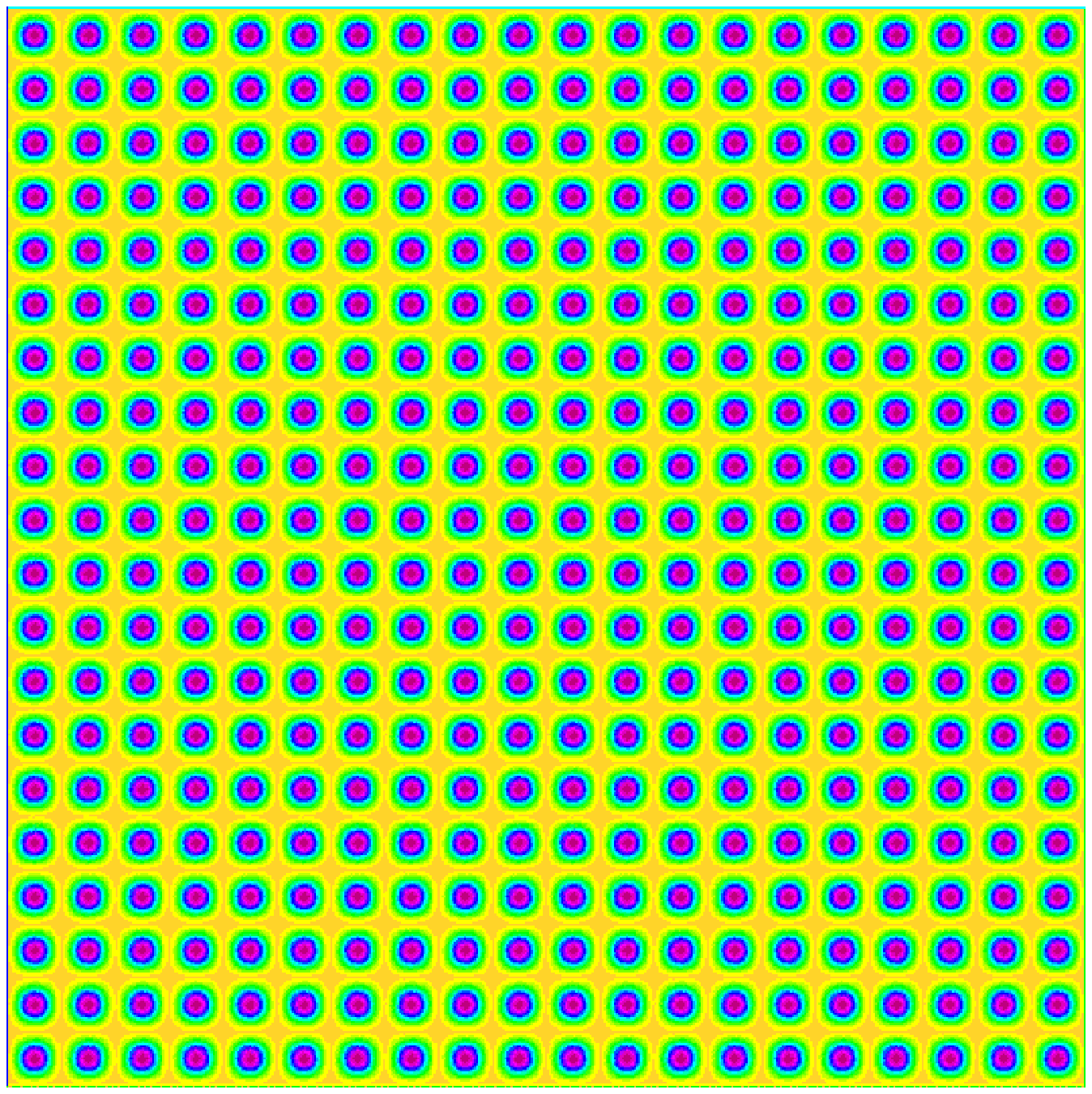}
\includegraphics[width=6cm]{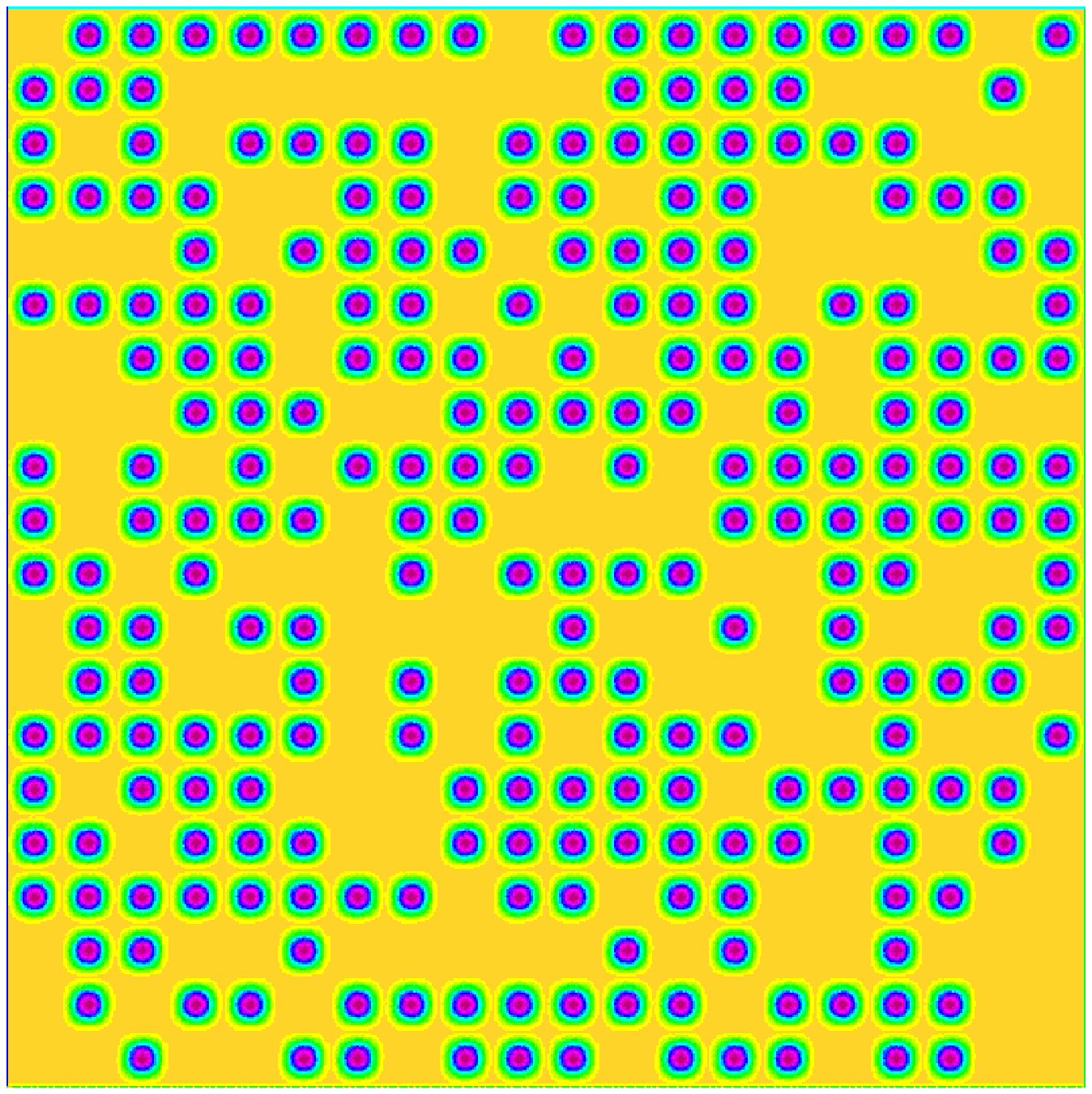}
\end{center}
\caption{Two instances of material~\eqref{eq:weak-random}. Left ($\eta=0$): perfect material with circular inclusions located on a periodic network. Right ($\eta=0.4$): perturbed material (each inclusion is deleted with a probability equal to 0.4). Courtesy A. Anantharaman and C. Le Bris.
\label{fig:cercles}
}
\end{figure}

\medskip

The above setting is actually quite general. Consider for instance a classical test-case, the random checkerboard case:
$$
A(x,\omega) = \sum_{k \in \mathbb{Z}^d}  \mathbf{1}_{Q + k}(x) X_k(\omega),
$$
where $X_k$ are i.i.d. random variables satisfying $\PP(X_k = \alpha) = \PP(X_k = \beta) = 1/2$. This model falls into the framework~\eqref{eq:weak-random}--\eqref{eq:random-form1}--\eqref{eq:ber} with
$$
A_{\rm per} = \alpha \, \Id, \quad C_{\rm per} = \beta \, \Id, \quad \eta = 1/2.
$$
An alternate choice (corresponding to choosing a different reference periodic materials) is 
$$
A_{\rm per} = \beta \, \Id, \quad C_{\rm per} = \alpha \, \Id, \quad \eta = 1/2.
$$
In this work, we restrict our attention to the case~\eqref{eq:weak-random}--\eqref{eq:random-form1}--\eqref{eq:ber}, i.e. when $B_k^\eta$ are i.i.d. Bernoulli random variables. This is the case specifically studied in~\cite{ALB_mms}. See~\cite{ALB_cras,ALB_ccp} for more general settings.

\subsection{Weakly-random homogenization result}
\label{sec:ALB_result}

Consider the model~\eqref{eq:weak-random}--\eqref{eq:random-form1}--\eqref{eq:ber}. The random variable $B_k^\eta(\omega)$ can take only two values, 0 or 1. Therefore, on the domain $Q_N$, there are only a finite number of realizations of $A_\eta(x,\omega)$. The realizations with the highest probability are as follows.
 
With probability $(1-\eta)^{|Q_N|}$, there are no defects in $Q_N$, and the realization actually corresponds to the perfect periodic situation. We introduce the periodic corrector $w_p^0$, solution to
\begin{equation} 
\label{eq:correcteur-per}
-\dive \left(A_{\rm per} \left( p +  \nabla w_p^0\right) \right) = 0,
\qquad
w_p^0 \ \mbox{is $Q$-periodic},
\end{equation}
and the associated matrix $A^\star_{\rm per}$, obtained by periodic homogenization:
\begin{equation}
\label{eq:A-per}
\forall p \in \RR^d, \quad
A^\star_{\rm per} \, p = 
\int_Q A_{\rm per} \left(p + \nabla w_p^0 \right).
\end{equation}
With probability $\eta (1-\eta)^{|Q_N|-1}$, there is a unique defect in $Q_N$, located, say, in the cell $k+Q$ (see Fig.~\ref{fig:1-2-defects}). Let us define
\begin{equation}
\label{eq:A1k}
A_1^k = A_{\rm per} + 1_{k+Q} \Big( C_{\rm per} - A_{\rm per} \Big),
\end{equation}
the associated corrector $w_p^{1,k,N}$, solution to
\begin{equation} 
\label{eq:correcteur-1N}
-\dive \left(A_1^k \left( p +  \nabla w_p^{1,k,N}\right) \right) = 0,
\qquad
w_p^{1,k,N} \ \mbox{is $Q_N$-periodic},
\end{equation}
and the homogenized matrix $A^\star_{1,k,N}$, given by
\begin{equation}
\label{eq:A-1N}
\forall p \in \RR^d, \quad
A^\star_{1,k,N} \, p = 
\frac{1}{|Q_N|} \int_{Q_N} A_1^k \left(p + \nabla w_p^{1,k,N} \right).
\end{equation}
With probability $\eta^2 (1-\eta)^{|Q_N|-2}$, there are two defects in $Q_N$, located, say, in the cells $k+Q$ and $l+Q$ (see Fig.~\ref{fig:1-2-defects}). Let us define
\begin{equation}
\label{eq:A2kl}
A_2^{k,l} = A_{\rm per} + \Big( 1_{k+Q} + 1_{l+Q} \Big) \, \Big( C_{\rm per} - A_{\rm per} \Big),
\end{equation}
the associated corrector $w_p^{2,k,l,N}$, solution to
\begin{equation} 
\label{eq:correcteur-2N}
-\dive \left(A_2^{k,l} \left( p +  \nabla w_p^{2,k,l,N}\right) \right) = 0,
\qquad
w_p^{2,k,l,N} \ \mbox{is $Q_N$-periodic},
\end{equation}
and the homogenized matrix $A^\star_{2,k,l,N}$, given by
\begin{equation}
\label{eq:A-2N}
\forall p \in \RR^d, \quad
A^\star_{2,k,l,N} \, p = 
\frac{1}{|Q_N|} \int_{Q_N} A_2^{k,l} \left(p + \nabla w_p^{2,k,l,N} \right).
\end{equation}
All the other configurations (with three defects or more) have a smaller probability.

\begin{figure}[htbp]
\centering{
\includegraphics[width=4cm]{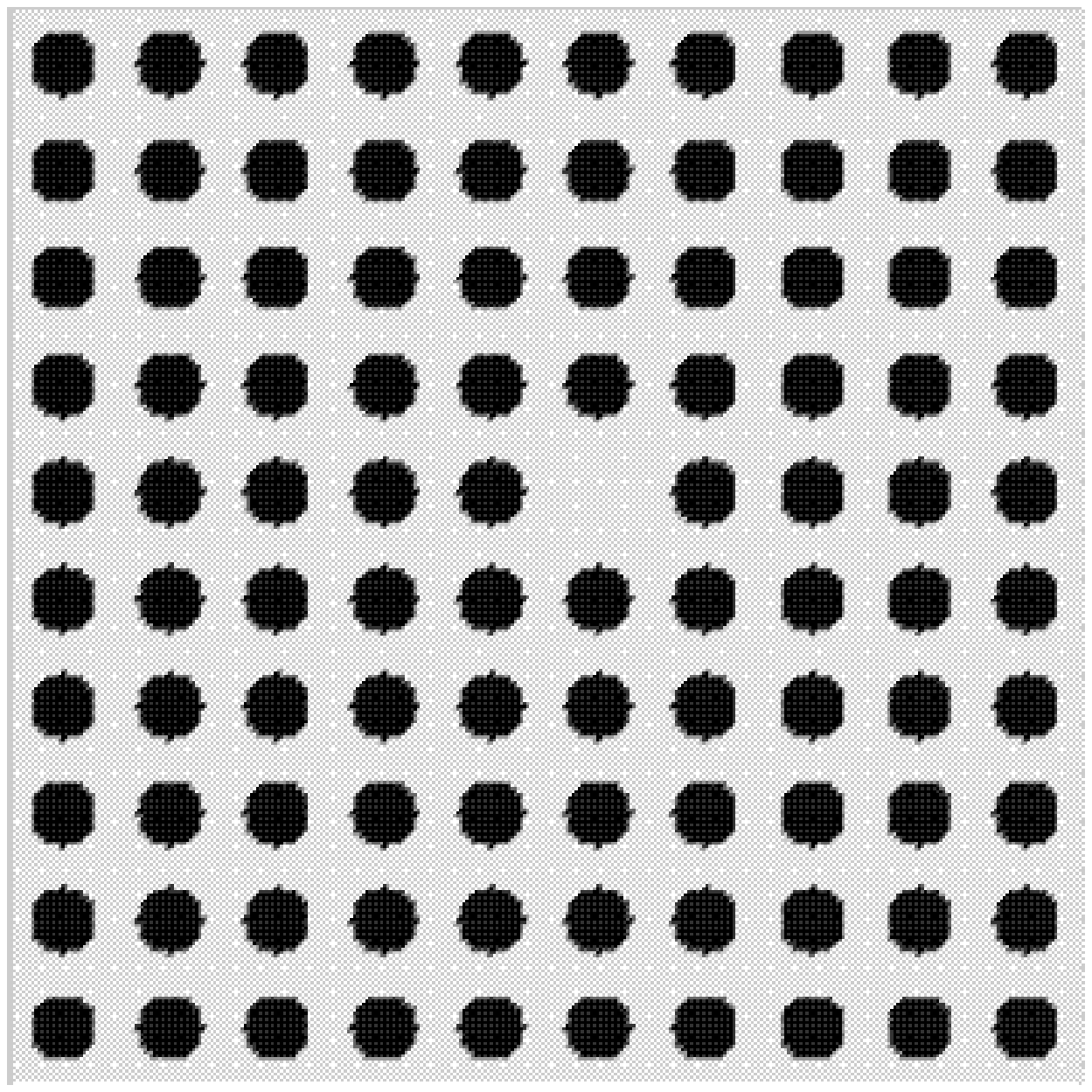}
\qquad
\includegraphics[width=4cm]{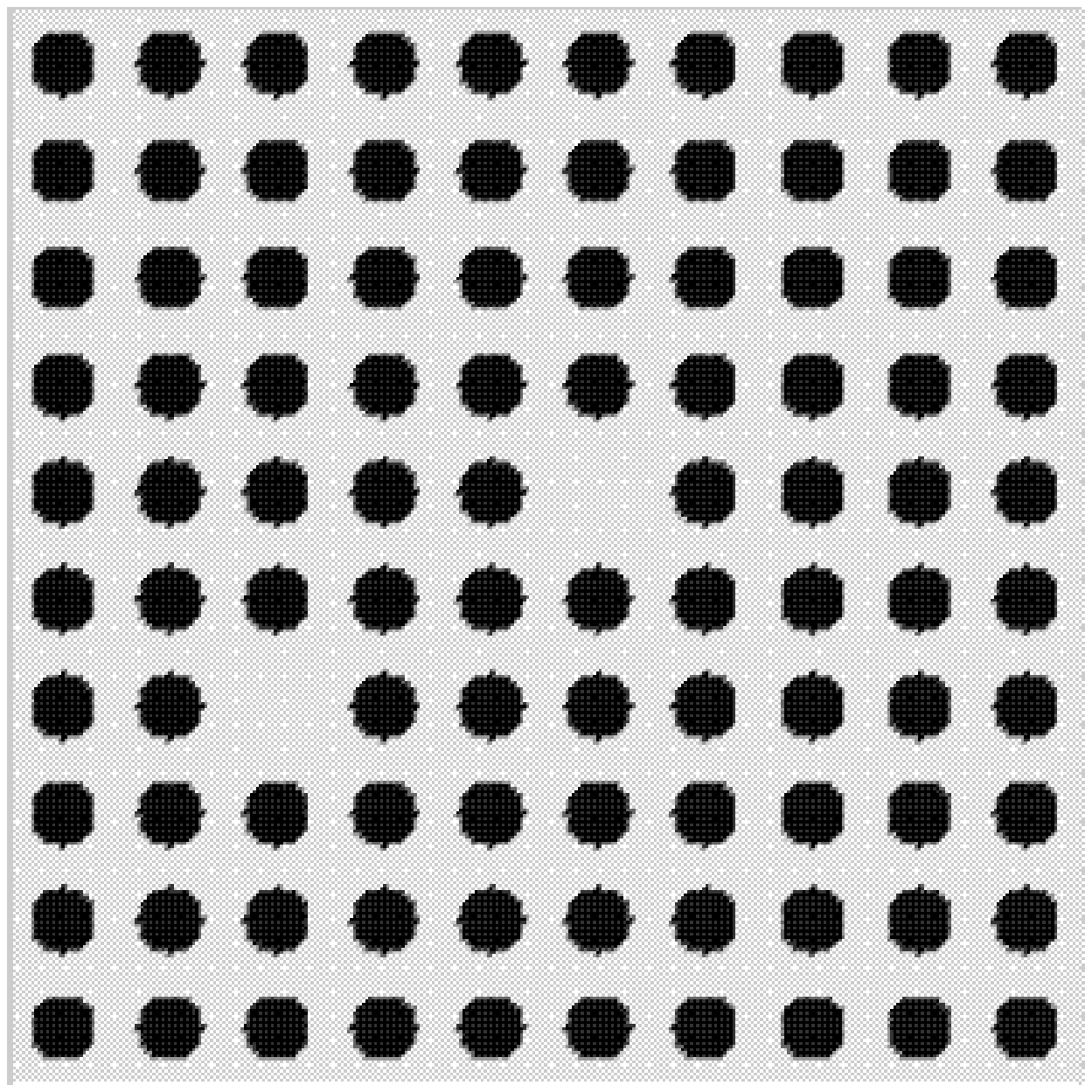}
}
\caption{Left: material modelled by $A_1^k$, with a single defect. Right: material modelled by $A_2^{k,l}$, with two defects (Courtesy A. Anantharaman and C. Le Bris).
\label{fig:1-2-defects}
}
\end{figure}

\medskip

Let us define
$$
\II_N := \left\{ k \in \ZZ^d; \ Q+k \subset Q_N \right\}.
$$
As shown in~\cite{ALB_mms}, we then have the following result:
\begin{proposition}[\cite{ALB_mms}, Section 3.2]
\label{weakly-random}
Let $A^\star_{\eta,N}(\omega)$ be the apparent homogenized matrix defined by~\eqref{eq:A-N}, where $A \equiv A_\eta$ is given by~\eqref{eq:weak-random}--\eqref{eq:random-form1}--\eqref{eq:ber}. Then
\begin{equation}
\label{eq:dl-weakly-random}
\EE \left[ A^\star_{\eta,N} \right] = A^\star_{\rm per} + \eta \overline{A}^N_1 + \eta^2 \overline{A}^N_2 + O_N(\eta^3),
\end{equation}
where $O_N(\eta^3)$ is a quantity of the order of $\eta^3$ with a prefactor that may depend on $N$, $A^\star_{\rm per}$ is given by~\eqref{eq:A-per} and 
\begin{eqnarray*}
\overline{A}^N_1 
&=& 
\sum_{k \in \II_N} \Big( A^\star_{1,k,N} - A^\star_{\rm per} \Big),
\\
\overline{A}^N_2 
&=& 
\frac12 \sum_{k, l \in \II_N, k \neq l} \Big( A^\star_{2,k,l,N} - A^\star_{1,k,N} - A^\star_{1,l,N} + A^\star_{\rm per} \Big).
\end{eqnarray*}
\end{proposition}

We note that
\begin{equation}
\label{eq:decompo}
\overline{A}^N_1  
= 
\sum_{k \in \II_N} \overline{A}^{k,N}_{\rm 1 \, def},
\quad \text{and} \quad
\overline{A}^N_2  
=
\frac12 \sum_{k \in \II_N} \sum_{l \in \II_N, l \neq k} \overline{A}^{k,l,N}_{\rm 2 \, def},
\end{equation}
where $\overline{A}^{k,N}_{\rm 1 \, def}$ (resp. $\overline{A}^{k,l,N}_{\rm 2 \, def}$) is the marginal contribution to the homogenized matrix from a configuration with a single defect in $k+Q$ (resp. two defects in $k+Q$ and $l+Q$):
\begin{eqnarray}
\label{eq:marg1}
\overline{A}^{k,N}_{\rm 1 \, def}
&=& 
A^\star_{1,k,N} - A^\star_{\rm per},
\\
\label{eq:marg2}
\overline{A}^{k,l,N}_{\rm 2 \, def}
&=& 
A^\star_{2,k,l,N} - A^\star_{1,k,N} - A^\star_{1,l,N} + A^\star_{\rm per}.
\end{eqnarray}

\begin{remark}
Passing to the limit $N \to \infty$ in~\eqref{eq:dl-weakly-random} is not easy. We refer to~\cite[Section 3.2]{ALB_mms} and~\cite{M13}.
\end{remark}

When $\eta$ is small, the advantage of~\eqref{eq:dl-weakly-random} over the approach recalled in Section~\ref{sec:bp} is evident. Rather than solving the random problem~\eqref{eq:correcteur-random-N} (for several realizations of $A_\eta$), it is enough to solve the deterministic problems~\eqref{eq:correcteur-per},~\eqref{eq:correcteur-1N} and~\eqref{eq:correcteur-2N} to infer an accurate approximation of $\EE \left[ A^\star_{\eta,N} \right]$. We refer to~\cite{ALB_mms} for illustrative numerical results. 

Furthermore, due to periodic boundary conditions~\eqref{eq:correcteur-1N}, that are reminiscent of the periodic boundary conditions in~\eqref{eq:correcteur-random-N}, we have that
\begin{equation}
\label{eq:thanks_PBC}
\text{$A^\star_{1,k,N}$ does not depend on $k$.} 
\end{equation}
Likewise, $A^\star_{2,k,l,N}$ depends only on $k-l$. Thus, there is only {\em one} problem~\eqref{eq:correcteur-1N} to be solved (say for $k=0$). Likewise, there are $|\II_N|-1$ problems~\eqref{eq:correcteur-2N} to be solved (say for $k=0$ and $l \neq 0$), and not $|\II_N| \, (|\II_N|-1$). Noticing that~\eqref{eq:correcteur-2N} is a problem parameterized by $l$, the authors of~\cite{clb_thomines} have shown how to use a Reduced Basis approach to further speed-up the computation of $\overline{A}^N_2$. In practice, one can still obtain a good approximation of $\overline{A}^N_2$ without solving all the $|\II_N|-1$ problems~\eqref{eq:correcteur-2N}. We return to this specific question in Section~\ref{sec:num_RB}.

\section{Control variate approaches for stochastic homogenization}
\label{sec:CV}

We now introduce, for the model~\eqref{eq:weak-random}--\eqref{eq:random-form1}--\eqref{eq:ber}, a control variate approach. Our aim is now to address the regime when $\eta$ is not close to 0 or 1 (the approximation~\eqref{eq:dl-weakly-random} is therefore not accurate enough). Recall also that, in view of the discussion at the end of Section~\ref{sec:CV_101}, we need a {\em random} surrogate model to build our controlled variable. In what follows, we first build an approximate model based on configurations with a single defect (see Section~\ref{sec:order-1}), and next turn to building a better approximate model that also uses configurations with two defects (see Section~\ref{sec:order-2}). As will be seen below, this second approximate model not only depends on the quantity of defects, but also on their geometry, that is on where the defects are located in $Q_N$.

\subsection{A first-order model}
\label{sec:order-1}

Introduce
\begin{equation}
\label{eq:Y1}
A^{\eta,N}_1(\omega)  
= 
\sum_{k \in \II_N} B_k^\eta(\omega) \, \overline{A}^{k,N}_{\rm 1 \, def},
\end{equation}
where $\overline{A}^{k,N}_{\rm 1 \, def}$, defined by~\eqref{eq:marg1}, is the marginal contribution to the homogenized matrix coming the configuration with a single defect located in $k+Q$. In view of~\eqref{eq:decompo}, we notice that
$$
\EE \left[ A^{\eta,N}_1 \right]
=
\sum_{k \in \II_N} \EE \left[ B_k^\eta \right] \overline{A}^{k,N}_{\rm 1 \, def}
=
\eta \sum_{k \in \II_N} \overline{A}^{k,N}_{\rm 1 \, def}
=
\eta \overline{A}^N_1,
$$
which is the first order correction in the expansion~\eqref{eq:dl-weakly-random}. When $\eta$ is small, the {\em expectation} of $A^\star_{\rm per} + A^{\eta,N}_1(\omega)$ is a good approximation of the expectation of $A^\star_{\eta,N}(\omega)$, accurate up to an error of the order of $\eta^2$. The following observation provides additional motivation for our choice~\eqref{eq:Y1}. It turns out that the {\em law} of the random variable $A^\star_{\rm per} + A^{\eta,N}_1(\omega)$ is a good approximation of that of $A^\star_{\eta,N}(\omega)$:

\begin{lemma}
\label{lem:approx_loi}
For any deterministic and continuous function $\varphi$, we have
$$
\EE \left[ \varphi \left( A^\star_{\eta,N} \right) \right]
=
\EE \left[ \varphi \left( A^\star_{\rm per} + A^{\eta,N}_1 \right) \right] + O_N(\eta^2).
$$
\end{lemma}
The proof of Lemma~\ref{lem:approx_loi} is postponed until Section~\ref{sec:theo_dD1}.

\medskip

We thus think that $A^\star_{\rm per} + A^{\eta,N}_1(\omega)$ is a good surrogate model for $A^\star_{\eta,N}(\omega)$. As shown by Lemma~\ref{lem:approx_loi}, this is the case when $\eta \ll 1$, which is however not the regime we address. One-dimensional computations presented in Section~\ref{sec:theo_1D} and numerical observations reported in Section~\ref{sec:num} (for two-dimensional test-cases) confirm that it is indeed the case, even when $\eta$ is not small. 

Following~\eqref{eq:control-variate-begin}, we now introduce our controlled variable as
\begin{eqnarray}
\nonumber
D_\rho^{1,\eta}(\omega) 
&=& 
A^\star_{\eta,N}(\omega) - \rho \left( A^\star_{\rm per} + A^{\eta,N}_1 (\omega) - \EE \left[ A^\star_{\rm per} + A^{\eta,N}_1 \right] \right)
\\
\label{eq:control-1}
&=&
A^\star_{\eta,N}(\omega) - \rho \left( A^{\eta,N}_1 (\omega) - \eta \overline{A}^N_1 \right).
\end{eqnarray}
In view of~\eqref{eq:Y1},~\eqref{eq:marg1} and~\eqref{eq:thanks_PBC}, we recast~\eqref{eq:control-1} as
\begin{equation}
\label{eq:control-1_bis}
D_\rho^{1,\eta}(\omega) 
=
A^\star_{\eta,N}(\omega) - \rho \left[ \left( \sum_{k \in \II_N} B_k^\eta(\omega) \right) - \eta \left| \II_N \right| \right] \overline{A}^{0,N}_{\rm 1 \, def}.
\end{equation} 

\begin{remark}
Note that, in~\eqref{eq:control-1_bis}, $A^\star_{\eta,N}(\omega)$ and $\sum_{k \in \II_N} B_k^\eta(\omega)$ are correlated. Indeed, in practice, we start by drawing a realization of the random variables $B^k_\eta(\omega)$ for all $k \in \II_N$. This determines first $\sum_{k \in \II_N} B_k^\eta(\omega)$, and second the field $A(x,\omega)$ on $Q_N$, from which we compute the associated $A^\star_{\eta,N}(\omega)$ following~\eqref{eq:A-N}--\eqref{eq:correcteur-random-N}. 
\end{remark}

Computing $M$ realizations of $D_\rho^{1,\eta}(\omega)$ therefore amounts to: 
\begin{itemize}
\item offline stage: determine $\overline{A}^{0,N}_{\rm 1 \, def}$ by solving the problem~\eqref{eq:correcteur-per}--\eqref{eq:A-per} on $Q$ and solving only once the problem~\eqref{eq:correcteur-1N}--\eqref{eq:A-1N} on $Q_N$ (say for $k=0$).
\item online stage: solve $M$ corrector problems~\eqref{eq:A-N}--\eqref{eq:correcteur-random-N} on $Q_N$ (for $M$ i.i.d. realizations of $A$ on $Q_N$), and evaluate $D_\rho^{1,\eta}(\omega)$ according to~\eqref{eq:control-1_bis}.
\end{itemize} 
Let ${\cal C}_N$ be the cost to solve a single corrector problem on $Q_N$. The Monte Carlo empirical estimator and the Control Variate empirical estimator, defined respectively by
$$
I^{\rm MC}_M = \frac{1}{M} \sum_{m=1}^M A^{\star,m}_{\eta,N}(\omega)
\quad 
\text{and}
\quad
I^{\rm CV}_M = \frac{1}{M} \sum_{m=1}^M D^{1,\eta,m}_\rho(\omega)
$$
therefore share the same cost ($M \, {\cal C}_N$ for the former, $(1+M) \, {\cal C}_N$ for the latter). To minimize the variance of $D^{1,\eta}_\rho$, the parameter $\rho$ in~\eqref{eq:control-1} is chosen following~\eqref{eq:rho_star}.

\bigskip

Notice that, in the above construction, we have considered as reference configuration the defect-free material, i.e. that for $\eta = 0$. Since, in the regime we focus on, $\eta$ is not small, there is no reason to favor the defect-free configuration ($\eta=0$) rather than the full defect configuration ($\eta=1$), which corresponds to the periodic matrix $C_{\rm per}$. We therefore introduce (compare with~\eqref{eq:marg1})
$$
\overline{C}^{k,N}_{\rm 1 \, def}
=
C^\star_{1,k,N} - C_{\rm per}^\star,
$$
where $C^\star_{1,k,N}$ is the homogenized matrix corresponding to a unique defect with respect to the periodic configuration $C_{\rm per}$ (compare with~\eqref{eq:A1k},~\eqref{eq:correcteur-1N} and~\eqref{eq:A-1N}):
\begin{equation}
\label{eq:D-1N}
\forall p \in \RR^d, \quad
C^\star_{1,k,N} \, p = 
\frac{1}{|Q_N|} \int_{Q_N} C_1^k \left(p + \nabla v_p^{1,k,N} \right),
\end{equation}
where, for any $p$, the corrector $v_p^{1,k,N}$ is a solution to
$$
-\dive \left(C_1^k \left( p +  \nabla v_p^{1,k,N}\right) \right) = 0,
\qquad
v_p^{1,k,N} \ \mbox{is $Q_N$-periodic},
$$
where $C_1^k = C_{\rm per} - 1_{k+Q} \Big( C_{\rm per} - A_{\rm per} \Big)$. In the spirit of~\eqref{eq:control-1_bis}, we introduce the controlled variable
$$
\widehat{D}_{\widehat{\rho}}^{1,\eta}(\omega) 
=
A^\star_{\eta,N}(\omega) - \widehat{\rho} \left[ \left( \sum_{k \in \II_N} (1-B_k^\eta(\omega)) \right) - (1-\eta) \left| \II_N \right| \right] \overline{C}^{0,N}_{\rm 1 \, def} ,
$$
that we recast as
$$
\widehat{D}_{\widehat{\rho}}^{1,\eta}(\omega) 
=
A^\star_{\eta,N}(\omega) + \widehat{\rho} \left[ \left( \sum_{k \in \II_N} B_k^\eta(\omega) \right) - \eta \left| \II_N \right| \right] \overline{C}^{0,N}_{\rm 1 \, def}.
$$
Consider now any entry $1 \leq i,j \leq d$ of the homogenized matrix. 
Assuming that our control variate model is non trivial (i.e. that $\left[ \overline{A}^{0,N}_{\rm 1 \, def} \right]_{ij} \neq 0$), we see that, for any deterministic $\widehat{\rho}$, there exists a deterministic parameter $\rho$ such that $\dis \left[ \widehat{D}_{\widehat{\rho}}^{1,\eta}(\omega) \right]_{ij} 
=
\left[ D_\rho^{1,\eta}(\omega) \right]_{ij}$ a.s. 
Working with the controlled variable $D_\rho^{1,\eta}(\omega)$ is hence equivalent to working with the controlled variable $\widehat{D}_{\widehat{\rho}}^{1,\eta}(\omega)$.  In the sequel, we only consider the former. 

\begin{remark}
The situation is different in the second order model, where taking $A_{\rm per}$ or $C_{\rm per}$ as reference is not equivalent. See Section~\ref{sec:order-2} below.
\end{remark}

\begin{remark}
In view of~\eqref{eq:control-1_bis}, we see that our first order control variable only depends on $\dis \sum_{k \in \II_N} B_k^\eta(\omega)$, which is the number of defects in the material. This approach can thus be extended to any two-phase materials, say of the type $A(x,\omega) = A_1 + \chi(x,\omega) A_2$, where $\chi$ is stationary and equal to 0 or 1. In this case, the control variable reads $\dis \int_{Q_N} \chi(x,\omega) \, dx$. We refer to~\cite{mb-fl} for works in that direction.
\end{remark}

\subsection{A second-order model}
\label{sec:order-2}

We now introduce a model that not only takes into account the contributions from single defects (through $\overline{A}^{k,N}_{\rm 1 \, def}$, see~\eqref{eq:Y1}) but also contributions from pairs of defects. To that aim, we introduce
\begin{equation}
\label{eq:Y2}
A^{\eta,N}_2(\omega)  
= 
\frac12 \sum_{k \in \II_N} \sum_{l \in \II_N, \, l \neq k} B_k^\eta(\omega) \, B_l^\eta(\omega) \, \overline{A}^{k,l,N}_{\rm 2 \, def},
\end{equation}
where $\overline{A}^{k,l,N}_{\rm 2 \, def}$, defined by~\eqref{eq:marg2}, is the marginal contribution to the homogenized matrix associated to the configuration with two defects located in $k+Q$ and $l+Q$. In view of~\eqref{eq:decompo}, we notice that
$$
\EE \left[ A^{\eta,N}_2 \right]
=
\frac12 \sum_{k \in \II_N} \sum_{l \in \II_N, \, l \neq k} \EE \left[ B_k^\eta \, B_l^\eta \right] \overline{A}^{k,l,N}_{\rm 2 \, def}
=
\frac{\eta^2}{2} \sum_{k \in \II_N} \sum_{l \in \II_N, \, l \neq k} \overline{A}^{k,l,N}_{\rm 2 \, def}
=
\eta^2 \overline{A}^N_2,
$$
which is the second order correction in the expansion~\eqref{eq:dl-weakly-random}. When $\eta$ is small, the {\em expectation} of $A^\star_{\rm per} + A^{\eta,N}_1(\omega) + A^{\eta,N}_2(\omega)$ is a good approximation of the expectation of $A^\star_{\eta,N}(\omega)$, accurate up to an error of the order of $\eta^3$. Furthermore, we have the following result (compare with Lemma~\ref{lem:approx_loi}), the proof of which follows the same lines as that of Lemma~\ref{lem:approx_loi} and is therefore omitted:

\begin{lemma}
\label{lem:approx_loi2}
For any deterministic and continuous function $\varphi$, we have
$$
\EE \left[ \varphi \left( A^\star_{\eta,N} \right) \right]
=
\EE \left[ \varphi \left( A^\star_{\rm per} + A^{\eta,N}_1 + A^{\eta,N}_2 \right) \right] + O_N(\eta^3).
$$
\end{lemma}

In a way similar to~\eqref{eq:control-1}, we now introduce our second-order controlled variable as
\begin{equation}
\label{eq:control-2}
D_{\rho_1,\rho_2}^{2,\eta}(\omega) 
= 
A^\star_{\eta,N}(\omega) - \rho_1 \left( A^{\eta,N}_1 (\omega) - \eta \overline{A}^N_1 \right) - \rho_2 \left( A^{\eta,N}_2 (\omega) - \eta^2 \overline{A}^N_2 \right).
\end{equation}
We have introduced two deterministic parameters $\rho_1$ and $\rho_2$, which need not be equal. For any choice of these parameters, we have $\EE \left[ D_{\rho_1,\rho_2}^{2,\eta} \right] = \EE \left[ A^\star_{\eta,N} \right]$. 

To evaluate~\eqref{eq:control-2}, we first have to precompute the deterministic matrices
$$
\overline{A}^{k,N}_{\rm 1 \, def} = \overline{A}^{0,N}_{\rm 1 \, def}
\quad \text{and} \quad
\overline{A}^{k,l,N}_{\rm 2 \, def} = \overline{A}^{0,l-k,N}_{\rm 2 \, def}.
$$
Computing $M$ realizations of $D_{\rho_1,\rho_2}^{2,\eta}(\omega)$ therefore amounts to:
\begin{itemize}
\item offline stage: (i) determine $\overline{A}^{0,N}_{\rm 1 \, def}$ by solving the problem~\eqref{eq:correcteur-per}--\eqref{eq:A-per} on $Q$ and by solving only once the problem~\eqref{eq:correcteur-1N}--\eqref{eq:A-1N} on $Q_N$ (say for $k=0$); (ii) determine $\overline{A}^{0,l,N}_{\rm 2 \, def}$ by solving $\left| \II_N \right|-1$ problems~\eqref{eq:correcteur-2N}--\eqref{eq:A-2N} on $Q_N$ (for $k=0$ and $l \in \II_N$, $l \neq 0$).
\item online stage: solve $M$ corrector problems~\eqref{eq:A-N}--\eqref{eq:correcteur-random-N} on $Q_N$ (for $M$ i.i.d. realizations of $A$ on $Q_N$), and evaluate $D_{\rho_1,\rho_2}^{2,\eta}(\omega)$ according to~\eqref{eq:control-2}.
\end{itemize} 
Questions related to the cost for evaluating $\overline{A}^{0,l,N}_{\rm 2 \, def}$ are discussed at the end of this section. 

\medskip

As pointed out in Section~\ref{sec:order-1}, in our regime of interest, there is no reason to favor the defect-free configuration rather than the full defect configuration, which corresponds to the periodic matrix $C_{\rm per}$. We have shown there that there is no use to introduce the terms representing the first order correction with respect to $C_{\rm per}$. We therefore solely introduce the second order correction (compare with~\eqref{eq:marg2}):
\begin{equation}
\label{eq:marg2_bis}
\overline{C}^{k,l,N}_{\rm 2 \, def}
=
C^\star_{2,k,l,N} - C^\star_{1,k,N} - C^\star_{1,l,N} + C_{\rm per}^\star,
\end{equation}
where $C^\star_{1,k,N}$ is defined by~\eqref{eq:D-1N} and $C^\star_{2,k,l,N}$ is defined by (compare with~\eqref{eq:A2kl}, \eqref{eq:correcteur-2N} and~\eqref{eq:A-2N}):
\begin{equation}
\label{eq:D-2N}
\forall p \in \RR^d, \quad
C^\star_{2,k,l,N} \, p = 
\frac{1}{|Q_N|} \int_{Q_N} C_2^{k,l} \left(p + \nabla v_p^{2,k,l,N} \right),
\end{equation}
where, for any $p \in \RR^d$, the corrector $v_p^{2,k,l,N}$ is a solution to
$$
-\dive \left(C_2^{k,l} \left( p +  \nabla v_p^{2,k,l,N}\right) \right) = 0,
\qquad
v_p^{2,k,l,N} \ \mbox{is $Q_N$-periodic},
$$
where $C_2^{k,l} = C_{\rm per} - \Big( 1_{k+Q} + 1_{l+Q} \Big) \, \Big( C_{\rm per} - A_{\rm per} \Big)$. As in~\eqref{eq:Y2}, we introduce
\begin{equation}
\label{eq:Y3}
C^{\eta,N}_2(\omega)  
= 
\frac12 \sum_{k \in \II_N} \sum_{l \in \II_N, \, l \neq k} \Big(1-B_k^\eta(\omega) \Big) \, \Big(1- B_l^\eta(\omega) \Big) \, \overline{C}^{k,l,N}_{\rm 2 \, def},
\end{equation}
where $\overline{C}^{k,l,N}_{\rm 2 \, def}$ is defined by~\eqref{eq:marg2_bis}, and its expectation reads
\begin{eqnarray*}
\overline{C}^{\eta,N}_2
:=
\EE \left[ C^{\eta,N}_2 \right]
&=& 
\frac12 \sum_{k \in \II_N} \sum_{l \in \II_N, \, l \neq k} \EE \left[ (1-B_k^\eta) \ (1- B_l^\eta) \right] \, \overline{C}^{k,l,N}_{\rm 2 \, def}
\\
&=& 
\frac12 \sum_{k \in \II_N} \sum_{l \in \II_N, \, l \neq k} (1-\eta)^2 \, \overline{C}^{k,l,N}_{\rm 2 \, def}.
\end{eqnarray*}
We eventually introduce the controlled variable (compare with~\eqref{eq:control-2})
\begin{multline}
\label{eq:control-3}
D_{\rho_1,\rho_2,\rho_3}^{3,\eta}(\omega) 
= 
A^\star_{\eta,N}(\omega) - \rho_1 \left( A^{\eta,N}_1 (\omega) - \eta \overline{A}^N_1 \right) 
\\
- \rho_2 \left( A^{\eta,N}_2 (\omega) - \eta^2 \overline{A}^N_2 \right) 
- \rho_3 \left( C^{\eta,N}_2 (\omega) - \overline{C}^{\eta,N}_2\right).
\end{multline}

Consider now a specific entry $1 \leq i,j \leq d$ of the homogenized matrix. The control variate approach consists in approximating $\EE \left[ \left( A^\star_{\eta,N} \right)_{ij} \right]$ by considering a Monte Carlo estimator for $\EE \left[ \left( D_{\rho_1,\rho_2,\rho_3}^{3,\eta} \right)_{ij} \right]$. The deterministic parameters $\rho_1$, $\rho_2$ and $\rho_3$ are chosen to minimize the variance of $\left( D_{\rho_1,\rho_2,\rho_3}^{3,\eta} (\omega) \right)_{ij}$. They are thus the solution of the following $3 \times 3$ linear system (we drop the subscript $i,j$ for conciseness): 
\begin{equation}
\label{eq:system_rho}
\begin{array}{rcl}
\Var[A^{\eta,N}_1] \rho_1 +  \Cov[A^{\eta,N}_1, A^{\eta,N}_2] \rho_2 + \Cov[A^{\eta,N}_1, C^{\eta,N}_2] \rho_3 
& = &
\Cov[A^\star_{\eta,N}, A^{\eta,N}_1]
\\
\Cov[A^{\eta,N}_2, A^{\eta,N}_1] \rho_1  + \Var[A^{\eta,N}_2] \rho_2 + \Cov[A^{\eta,N}_2, C^{\eta,N}_2] \rho_3 
& = &
\Cov[A^\star_{\eta,N}, A^{\eta,N}_2]
\\
\Cov[C^{\eta,N}_2, A^{\eta,N}_1] \rho_1  + \Cov[C^{\eta,N}_2, A^{\eta,N}_2] \rho_2 + \Var[C^{\eta,N}_2] \rho_3 
& = &
\Cov[A^\star_{\eta,N}, C^{\eta,N}_2] 
\end{array}
\end{equation}
depending on the covariances between the entries $ij$ of $A^\star_{\eta,N}$, $A^{\eta,N}_1$, $A^{\eta,N}_2$ and $C^{\eta,N}_2$. In practice, these covariances are approximated by empirical estimators (see Remark~\ref{rem:choix_rho}). 

\medskip

In practice, computing the matrices $\overline{A}^{0,l,N}_{\rm 2 \, def}$ (and likewise $\overline{C}^{0,l,N}_{\rm 2 \, def}$) is rather expensive (because each problem is set on the large domain $Q_N$, and the number of these problems increases when $N$ increases). It is therefore useful to approximate them using the Reduced Basis strategy introduced in~\cite{clb_thomines}, which dramatically decreases the computational cost. The procedure is essentially as follows. We first solve the single defect problem~\eqref{eq:correcteur-1N} for $k=0$, and solve~\eqref{eq:correcteur-2N} for a limited number of locations of the defect pairs, say $k=0$ and $l$ close to $k$. On the basis of these computations, we are then in position to obtain very efficient approximations of the matrices $\overline{A}^{0,l,N}_{2 \ \rm def}$ for all $l \in \II_N$, $l \neq 0$. Evaluating~\eqref{eq:Y2} is thus inexpensive. Thus, up to a limited offline cost (i.e. the cost for solving the few problems~\eqref{eq:correcteur-2N} that we have to consider), the Monte Carlo empirical estimator and the Control Variate empirical estimator, defined respectively by
$$
I^{\rm MC}_M = \frac{1}{M} \sum_{m=1}^M A^{\star,m}_{\eta,N}(\omega)
\quad
\text{and}
\quad
I^{\rm CV}_M := \frac{1}{M} \sum_{m=1}^M D_{\rho_1,\rho_2,\rho_3}^{3,\eta,m}(\omega) 
$$
share the same cost. We refer to Section~\ref{sec:num_RB} for numerical experiments using this procedure.

\begin{remark}
In sharp contrast to the first order control variable, the second order control variable not only depends on the number of defects in the materials, i.e. $\dis \sum_{k \in \II_N} B_k^\eta(\omega)$, but also on their location. The specific geometry of the materials, which is ignored in~\eqref{eq:control-1_bis}, is taken into account in~\eqref{eq:control-3}.
\end{remark}

\section{Elements of theoretical analysis}
\label{sec:theo}

This section is devoted to establishing estimates on the gain provided by our approach. We proceed in two directions. First, in Section~\ref{sec:theo_1D}, we consider the one-dimensional case. Our main results are Propositions~\ref{prop:theo_1D} and~\ref{prop:theo_1D_suite}. We consider the large $N$ regime, and estimate the variance (in terms of $N$) of $A^\star_{\eta,N}$, the controlled variables $D_\rho^{1,\eta}$ defined by~\eqref{eq:control-1} and $D_{\rho_1,\rho_2,\rho_3}^{3,\eta}$ defined by~\eqref{eq:control-3}. We show that they are of the order of $N^{-1}$, $N^{-2}$ and $N^{-3}$, respectively. Note that, in this section, we do not assume $\eta$ to be close to 0 or 1, i.e. we are in a fully random case. 

In Section~\ref{sec:theo_dD}, we turn to the multi-dimensional case. Our main result is Lemma~\ref{lem:variance-scaling}. We consider the regime when $\eta$ is small, and estimate the variance (in terms of $\eta$) of $A^\star_{\eta,N}$ and of the controlled variables $D_\rho^{1,\eta}$ defined by~\eqref{eq:control-1} and $D_{\rho_1,\rho_2}^{2,\eta}$ defined by~\eqref{eq:control-2}. We show that the control variate approach using the first order (resp. second order) surrogate model allows to decrease the variance from $O(\eta)$ to $O(\eta^2)$ (resp. from $O(\eta)$ to $O(\eta^3)$). 

Still in the regime $\eta \ll 1$, we show in Section~\ref{sec:theo_dD3} that, for an equal computational cost, the weakly stochastic approach proposed in~\cite{ALB_mms} (which directly compute $\EE(A^\star_{\eta,N})$ as in series in powers of $\eta$) is more accurate than the control variate approach proposed in this work. The regime of interest for our approach is therefore when $\eta$ is neither close to 0 nor to 1. This is the regime we consider in the numerical experiments of Section~\ref{sec:num}.

\subsection{One-dimensional case}
\label{sec:theo_1D}

In the one-dimensional case, we know that
$$
A^\star_{\eta,N}(\omega)
= 
\left( \frac{1}{N} \int_0^N \frac{1}{A_\eta(x,\omega)} \right)^{-1},
$$
where, for ease of notation, we set $Q_N=(0,N)$ rather than $Q_N=(-N/2,N/2)$ as before. In view of~\eqref{eq:weak-random}--\eqref{eq:random-form1}--\eqref{eq:ber}, we thus have
$$
\frac{1}{A^\star_{\eta,N}(\omega)}
= 
\frac{1}{N} \sum_{k=0}^{N-1} \int_k^{k+1} \frac{dx}{A_{\rm per}(x) + B_k^\eta(\omega) \Big(C_{\rm per}(x) - A_{\rm per}(x)\Big)}.
$$
Introducing the functions
$$
f(x) = \frac{1}{x} 
\quad \text{and} \quad
\phi(b) = \int_0^1 \frac{dx}{A_{\rm per}(x) + b \Big(C_{\rm per}(x) - A_{\rm per}(x)\Big)},
$$
we thus see that
$$
A^\star_{\eta,N}(\omega) = f\left(\frac{1}{N} \sum_{k=0}^{N-1} \phi(B_k^\eta(\omega)) \right).
$$
Since $B_k^\eta(\omega)$ are equal to 0 or 1, we can write $\phi(B_k^\eta(\omega)) = \phi(0)+ B_k^\eta(\omega) (\phi(1)-\phi(0))$, and thus 
\begin{equation}
\label{eq:astar-1D}
A^\star_{\eta,N}(\omega) = g\left(\frac{1}{N} \sum_{k=0}^{N-1} B_k^\eta(\omega) \right)
\end{equation}
where the smooth function $g$ is defined by $g(b) = f\Big(\phi(0)+ b (\phi(1)-\phi(0)) \Big)$.

\subsubsection{First order model}

In view of~\eqref{eq:control-1}, \eqref{eq:Y1} and~\eqref{eq:marg1}, the first-order surrogate model is given by $A^\star_{\rm per} + A^{\eta,N}_1(\omega)$, with
\begin{equation}
\label{eq:gigi2}
A^{\eta,N}_1(\omega)  
= 
\sum_{k=0}^{N-1} B_k^\eta(\omega) \, \overline{A}^{k,N}_{\rm 1 \, def}
= 
\overline{A}^{0,N}_{\rm 1 \, def} \sum_{k=0}^{N-1} B_k^\eta(\omega). 
\end{equation}
We first state the following general result, the proof of which is postponed until Section~\ref{sec:1D_proof}.
\begin{lemma}
\label{lem:theo_1D}
Let
$$
X(\omega) = g\left(\frac{1}{N} \sum_{k=0}^{N-1} B_k(\omega) \right)
$$
where $B_k(\omega)$ are i.i.d. random variables valued in $[0,1]$ and $g$ is a function in $C^3(\RR)$. Then 
\begin{equation}
\label{eq:resu1}
\Var(X) = \frac{(g'(\eta))^2 \, \sigma^2}{N} + O \left( \frac{1}{N^2} \right)
\end{equation}
with $\eta = \EE(B_0)$ and $\sigma = \sqrt{\Var(B_0)}$.

For any $\rho$, introduce 
\begin{equation}
\label{eq:def_Y1}
D_\rho(\omega) = X(\omega) - \rho \Big( Y_1(\omega) - \EE \left[ Y_1 \right] \Big)
\quad \text{where} \quad
Y_1(\omega) = \sum_{k=0}^{N-1} B_k(\omega).
\end{equation}
There exists a constant $C$ independent of $N$ and some deterministic parameter $\rho_N$ such that 
\begin{equation}
\label{eq:resu2}
\Var \left( D_{\rho_N} \right) \leq \frac{C}{N^2}.
\end{equation}
\end{lemma}

The following proposition, of direct interest to us, directly falls from the above lemma.
\begin{proposition}
\label{prop:theo_1D}
Consider the model~\eqref{eq:weak-random}--\eqref{eq:random-form1}--\eqref{eq:ber}. Let $A^\star_{\eta,N}$ be the apparent homogenized matrix defined by~\eqref{eq:A-N}--\eqref{eq:correcteur-random-N} and $D^{1,\eta}_\rho$ be the first-order controlled variable defined by~\eqref{eq:control-1}. In the one-dimensional case, we have
\begin{equation}
\label{eq:resu1_specific}
\Var(A^\star_{\eta,N}) = \frac{C}{N} + O \left( \frac{1}{N^2} \right)
\end{equation}
and, for the optimal value of the deterministic parameter $\rho$,
\begin{equation}
\label{eq:resu2_specific}
\min_\rho \Var \Big( D^{1,\eta}_\rho \Big) = \Var \Big( D^{1,\eta}_{\rho^\star} \Big) = O \left( \frac{1}{N^2} \right).
\end{equation}
\end{proposition}

Using the control variate approach based on the first-order model, the variance is thus improved by at least one order in terms of $N$. Note in particular that, in the above results, we have not assumed $\eta$ to be small.

\begin{proof}[Proof of Proposition~\ref{prop:theo_1D}]
The proof of~\eqref{eq:resu1_specific} falls from~\eqref{eq:astar-1D} and~\eqref{eq:resu1}. We now prove~\eqref{eq:resu2_specific}. In view of~\eqref{eq:control-1}, \eqref{eq:astar-1D}, \eqref{eq:gigi2} and~\eqref{eq:def_Y1}, we see that
$$
D^{1,\eta}_\rho 
=
X(\omega) - \rho \overline{A}^{0,N}_{\rm 1 \, def} \Big( Y_1(\omega) - \EE \left[ Y_1 \right] \Big).
$$
Using~\eqref{eq:resu2}, we thus have
$$
\min_\rho \Var \Big( D^{1,\eta}_\rho \Big) 
\leq
\Var \Big( D_{\rho_N} \Big)
\leq
\frac{C}{N^2},
$$
which concludes the proof of Proposition~\ref{prop:theo_1D}.
\end{proof}

\subsubsection{Second order model}

In view of~\eqref{eq:control-3}, \eqref{eq:Y1}, \eqref{eq:Y2} and~\eqref{eq:Y3}, the second-order controlled variable reads
\begin{multline*}
D_{\rho_1,\rho_2,\rho_3}^{3,\eta}(\omega) 
= 
A^\star_{\eta,N}(\omega) - \rho_1 \overline{A}^{0,N}_{\rm 1 \, def} \sum_{k=0}^{N-1} \Big( B_k^\eta(\omega) - \eta \Big)
\\
- \rho_2 \overline{A}^{0,1,N}_{\rm 2 \, def} \sum_{k \neq l}^{N-1} \Big( B_k^\eta(\omega) B_l^\eta(\omega) - \eta^2 \Big)
\\
- \rho_3 \overline{C}^{0,1,N}_{\rm 2 \, def} \sum_{k \neq l}^{N-1} \Big( (1-B_k^\eta(\omega)) (1- B_l^\eta(\omega)) - (1-\eta)^2 \Big)
\end{multline*}
where we have used~\eqref{eq:thanks_PBC} and the fact that, in the one-dimensional case, $\overline{A}^{k,l,N}_{\rm 2 \, def}$ and $\overline{C}^{k,l,N}_{\rm 2 \, def}$ are independent of $k$ and $l$. We hence obtain that
\begin{equation}
\label{eq:gigi4}
D_{\rho_1,\rho_2,\rho_3}^{3,\eta}(\omega) 
= 
A^\star_{\eta,N}(\omega) - \overline{\rho_1} \sum_{k=0}^{N-1} \Big( B_k^\eta(\omega) - \eta \Big)
-  \overline{\rho_2} \sum_{k \neq l}^{N-1} \Big( B_k^\eta(\omega) B_l^\eta(\omega) - \eta^2 \Big)
\end{equation}
with
$$
\overline{\rho_1}
=
\rho_1 \overline{A}^{0,N}_{\rm 1 \, def} 
-
2 (N-1) \rho_3 \overline{C}^{0,1,N}_{\rm 2 \, def}, 
\qquad
\overline{\rho_2}
=
\rho_2 \overline{A}^{0,1,N}_{\rm 2 \, def} 
+ 
\rho_3 \overline{C}^{0,1,N}_{\rm 2 \, def}. 
$$
We first state the following general result, the proof of which is postponed until Section~\ref{sec:1D_proof}.
\begin{lemma}
\label{lem:theo_1D_suite}
Let
$$
X(\omega) = g\left(\frac{1}{N} \sum_{k=0}^{N-1} B_k(\omega) \right)
$$
where $g$ is a function in $C^3(\RR)$ and $B_k(\omega)$ are i.i.d. random variables taking values in $\{0,1\}$. Let $Y_1$ be defined by~\eqref{eq:def_Y1} and $Y_2$ be defined by
\begin{equation}
\label{eq:def_Y2}
Y_2(\omega) = \sum_{k=0}^{N-1} \sum_{l=0, l \neq k}^{N-1} B_k(\omega) B_l(\omega).
\end{equation}
There exists a constant $C$ independent of $N$ and some deterministic parameters $\overline{\rho_1}$ and $\overline{\rho_2}$ (that depend on $N$) such that 
\begin{equation}
\label{eq:resu3}
\Var \Big( \overline{D}_{\overline{\rho_1},\overline{\rho_2}} \Big) \leq \frac{C}{N^3}
\end{equation}
where $\dis
\overline{D}_{\overline{\rho_1},\overline{\rho_2}}(\omega)
= 
X(\omega) - \overline{\rho_1} \Big( Y_1(\omega) - \EE(Y_1) \Big) - \overline{\rho_2} \Big( Y_2(\omega) - \EE(Y_2) \Big)$.
\end{lemma}

The following proposition directly falls from the above lemma.
\begin{proposition}
\label{prop:theo_1D_suite}
Consider the model~\eqref{eq:weak-random}--\eqref{eq:random-form1}--\eqref{eq:ber}. Let $A^\star_{\eta,N}$ be the apparent homogenized matrix defined by~\eqref{eq:A-N}--\eqref{eq:correcteur-random-N} and $D_{\rho_1,\rho_2,\rho_3}^{3,\eta}(\omega)$ be the second-order controlled variable defined by~\eqref{eq:control-3}. In the one-dimensional case, for the optimal value of the deterministic parameters $\rho_1$, $\rho_2$ and $\rho_3$, we have
\begin{equation}
\label{eq:resu3_specific}
\min_{\rho_1,\rho_2,\rho_3} \Var \Big( D^{3,\eta}_{\rho_1,\rho_2,\rho_3} \Big) = O \left( \frac{1}{N^3} \right).
\end{equation}
\end{proposition}

We recall that 
$$
\Var \Big(A^\star_{\eta,N} \Big) = \frac{C}{N} + O \left( \frac{1}{N^2} \right).
$$
Thus, using the control variate approach based on the second-order model, the variance is improved by at least two orders in terms of $N$. This result is to be compared with Proposition~\ref{prop:theo_1D}.

\begin{proof}[Proof of Proposition~\ref{prop:theo_1D_suite}]
In view of~\eqref{eq:gigi4}, \eqref{eq:astar-1D}, \eqref{eq:def_Y1} and~\eqref{eq:def_Y2}, we see that
$$
D_{\rho_1,\rho_2,\rho_3}^{3,\eta}(\omega) 
= 
X(\omega) - \overline{\rho_1} \Big( Y_1(\omega) - \EE \left[ Y_1 \right] \Big)
-  \overline{\rho_2} \Big( Y_2(\omega) - \EE \left[ Y_2 \right] \Big).
$$
Using~\eqref{eq:resu3}, we thus have
$$
\min_{\rho_1,\rho_2,\rho_3} \Var \Big( D^{3,\eta}_{\rho_1,\rho_2,\rho_3} \Big)
\leq
\Var \Big( \overline{D}_{\overline{\rho_1},\overline{\rho_2}} \Big)
\leq
\frac{C}{N^3},
$$
which concludes the proof of Proposition~\ref{prop:theo_1D_suite}.
\end{proof}

\subsubsection{Proofs of Lemmas~\ref{lem:theo_1D} and~\ref{lem:theo_1D_suite}}
\label{sec:1D_proof}

\begin{proof}[Proof of Lemma~\ref{lem:theo_1D}]
Introducing the centered random variables
$$
d_k(\omega) = B_k(\omega) - \eta
$$
and a smooth function $h$ on $[0,1]$, we write 
\begin{eqnarray}
h \left( \frac{1}{N} \sum_{k=0}^{N-1} B_k(\omega) \right)
&=&
h \left( \eta + \frac{1}{N} \sum_{k=0}^{N-1} d_k(\omega) \right)
\nonumber
\\
&=&
h(\eta) + \frac{h'(\eta)}{N} \sum_{k=0}^{N-1} d_k(\omega) 
+ \frac{h''(\eta)}{2} \left( \frac{1}{N} \sum_{k=0}^{N-1} d_k(\omega) \right)^2
\nonumber
\\
&& \qquad \qquad \qquad 
+ \frac{h'''(\theta_3^N(\omega))}{6} \left( \frac{1}{N} \sum_{k=0}^{N-1} d_k(\omega) \right)^3 
\label{eq:gigi}
\end{eqnarray}
for some $\theta_3^N(\omega) \in [0,1]$. Recall now that any i.i.d. variables $d_k$ with mean value zero satisfy the following bounds:
\begin{equation}
\label{eq:fondam}
\forall p \in \NN^\star, \ \exists C_p > 0, \quad 
\left| \EE \left[ 
\left( \frac{1}{N} \sum_{k=0}^{N-1} d_k \right)^p
\right] \right|
\leq
\left\{
\begin{array}{c}
\dis \frac{C_p}{N^{p/2}} \text{ if $p$ is even;}
\\ \noalign{\vskip 3pt}
\dis \frac{C_p}{N^{(p+1)/2}} \text{ if $p$ is odd.}
\end{array}
\right.
\end{equation}
This is proved by developing the power $p$ of the sum, and then using the fact that the variables are i.i.d and have mean value zero. Taking expectations in~\eqref{eq:gigi}, we thus deduce that
$$
\EE \left[ 
h \left( \frac{1}{N} \sum_{k=0}^{N-1} B_k(\omega) \right)
\right]
=
h(\eta) 
+ \frac{h''(\eta)}{2N} \sigma^2 
+ O \left( \frac{1}{N^2} \right), 
$$
where $\sigma^2 = \EE[ d_0^2 ] = \Var(B_0)$. Choosing $h(x) = g(x)$ and $h(x) = (g(x))^2$, we obtain~\eqref{eq:resu1}.

\medskip

We next turn to proving~\eqref{eq:resu2}. As in~\eqref{eq:gigi}, we have
\begin{eqnarray*}
X(\omega)
&=&
g \left( \frac{1}{N} \sum_{k=0}^{N-1} B_k(\omega) \right)
\\
&=&
g(\eta) + \frac{g'(\eta)}{N} \sum_{k=0}^{N-1} d_k(\omega) 
+ \frac{g''(\theta_2^N(\omega))}{2} \, S_N(\omega)
\\
&=&
g(\eta) + \frac{g'(\eta)}{N} \Big( Y_1(\omega) - \EE \left[ Y_1 \right] \Big)
+ \frac{g''(\theta_2^N(\omega))}{2} \, S_N(\omega)
\end{eqnarray*}
for some $\theta_2^N(\omega) \in [0,1]$, where $\dis S_N(\omega) = \left( \frac{1}{N} \sum_{k=0}^{N-1} d_k(\omega) \right)^2$. Set $\dis \rho_N = \frac{g'(\eta)}{N}$. Then
$$
D_{\rho_N}(\omega) 
= 
X(\omega) - \rho_N \Big( Y_1(\omega) - \EE \left[ Y_1 \right] \Big)
=
g(\eta) + \frac{g''(\theta_2^N(\omega))}{2} \, S_N(\omega).
$$
Using~\eqref{eq:fondam}, we thus obtain that
$$
\Var \left( D_{\rho_N} \right) 
\leq 
\EE \left[ \left( \frac{g''(\theta_2^N(\omega))}{2} \, S_N(\omega) \right)^2 \right]
\leq
C \EE \left[ (S_N)^2 \right]
\leq
\frac{C}{N^2}
$$
which is the claimed bound~\eqref{eq:resu2}. This concludes the proof of Lemma~\ref{lem:theo_1D}. 
\end{proof}

\begin{proof}[Proof of Lemma~\ref{lem:theo_1D_suite}]
We follow the same lines as in the proof of Lemma~\ref{lem:theo_1D}. 
Introducing the centered random variables
$$
d_k(\omega) = B_k(\omega) - \eta,
$$
we write, as in~\eqref{eq:gigi}, that
\begin{eqnarray}
X(\omega)
&=&
g \left( \frac{1}{N} \sum_{k=0}^{N-1} B_k(\omega) \right)
\nonumber
\\
&=&
g(\eta) + \frac{g'(\eta)}{N} \sum_{k=0}^{N-1} d_k(\omega) 
+ \frac{g''(\eta)}{2N^2} \left( \sum_{k=0}^{N-1} d_k(\omega) \right)^2 
+ \frac{g'''(\theta_3^N(\omega))}{6} S_N(\omega)
\label{eq:gigi3}
\end{eqnarray}
for some $\theta_3^N(\omega) \in [0,1]$, where $\dis S_N(\omega) = \left( \frac{1}{N} \sum_{k=0}^{N-1} d_k(\omega) \right)^3$. We now recall that $\dis \sum_{k=0}^{N-1} d_k(\omega) = Y_1 - \EE(Y_1)$. Furthermore, we compute that
$$
\left( \sum_{k=0}^{N-1} d_k(\omega) \right)^2 
=
N^2 \eta^2 + (1-2N\eta) Y_1(\omega) + Y_2(\omega).
$$
We thus recast~\eqref{eq:gigi3} as
$$
X(\omega)
=
{\cal C} + \frac{g'(\eta)}{N} Y_1(\omega) 
+ \frac{g''(\eta)}{2N^2} \Big( (1-2N\eta) Y_1(\omega) + Y_2(\omega) \Big)
+ \frac{g'''(\theta_3^N(\omega))}{6} S_N(\omega)
$$
where ${\cal C}$ is a deterministic quantity.

Set $\dis \overline{\rho_1} = \frac{g'(\eta)}{N} + \frac{g''(\eta)}{2N^2} (1-2N\eta)$ and $\dis \overline{\rho_2} = \frac{g''(\eta)}{2N^2}$. Then
$$
\overline{D}_{\overline{\rho_1},\overline{\rho_2}}(\omega)
= 
X(\omega) - \overline{\rho_1} \Big( Y_1(\omega) - \EE(Y_1) \Big) - \overline{\rho_2} \Big( Y_2(\omega) - \EE(Y_2) \Big)
=
{\cal C} + \frac{g'''(\theta_3^N(\omega))}{6} \, S_N(\omega).
$$
Using~\eqref{eq:fondam}, we thus obtain that
$$
\Var \left( \overline{D}_{\overline{\rho_1},\overline{\rho_2}}\right) 
\leq 
\EE \left[ \left( \frac{g'''(\theta_3^N(\omega))}{6} \, S_N(\omega) \right)^2 \right]
\leq
C \EE \left[ (S_N)^2 \right]
\leq
\frac{C}{N^3}
$$
which is the claimed bound~\eqref{eq:resu3}. This concludes the proof of Lemma~\ref{lem:theo_1D_suite}.
\end{proof}

\subsection{Multi-dimensional case}
\label{sec:theo_dD}

\subsubsection{Proof of Lemma~\ref{lem:approx_loi}}
\label{sec:theo_dD1}

The proof follows the same lines as that of~\eqref{eq:dl-weakly-random}. It falls by enumerating the possible configurations according to the number of defects they include. We thus have, following Section~\ref{sec:ALB_result},
\begin{equation}
\label{eq:toto}
\EE \left[ \varphi \left( A^\star_{\eta,N} \right) \right]
=
(1-\eta)^{|Q_N|} \varphi \left( A^\star_{\rm per} \right)
+
\sum_{k \in \II_N} \eta (1-\eta)^{|Q_N|-1} \varphi \left( A^\star_{1,k,N} \right)
+
O_N(\eta^2).
\end{equation}
On the other hand, using~\eqref{eq:Y1} and~\eqref{eq:marg1}, we write
\begin{eqnarray*}
&&
\EE \left[ \varphi \left( A^\star_{\rm per} + A^{\eta,N}_1 \right) \right]
\\
&=&
(1-\eta)^{|Q_N|} \varphi \left( A^\star_{\rm per} \right)
+
\sum_{k \in \II_N} \eta (1-\eta)^{|Q_N|-1} \varphi \left( A^\star_{\rm per} + \overline{A}^{k,N}_{\rm 1 \, def} \right)
+
O_N(\eta^2)
\\
&=&
(1-\eta)^{|Q_N|} \varphi \left( A^\star_{\rm per} \right)
+
\sum_{k \in \II_N} \eta (1-\eta)^{|Q_N|-1} \varphi \left( A^\star_{1,k,N} \right)
+
O_N(\eta^2).
\end{eqnarray*}
We deduce from the above relation and~\eqref{eq:toto} the claimed result.

\subsubsection{Estimates of the variances as a function of $\eta$}
\label{sec:theo_dD2}

Lemmas~\ref{lem:approx_loi} and~\ref{lem:approx_loi2} show that our surrogate model is a good approximation (in terms of its law) of the random variable $A^\star_{\eta,N}$. The lemma below shows, again in the regime $\eta \ll 1$, that variance is indeed decreased. 

\medskip

Consider any entry $ij$ of the homogenized matrix. The estimation of $\EE \left[ \left( A^\star_{\eta,N} \right)_{ij} \right]$ can be done by a Monte Carlo empirical mean on $\left( A^\star_{\eta,N}(\omega) \right)_{ij}$, $\left( D^{1,\eta}_{\rho}(\omega) \right)_{ij}$ (see Section~\ref{sec:order-1}) or $\left( D^{2,\eta}_{\rho_1,\rho_2}(\omega) \right)_{ij}$ (see Section~\ref{sec:order-2}).
 
\begin{lemma}
\label{lem:variance-scaling}
For any entry $ij$ of the homogenized matrix, we have 
\begin{eqnarray}
\label{eq:tutu1}
\Var \left[ \left( A^\star_{\eta,N} \right)_{ij} \right] &=& \eta C^0_N + O_N(\eta^2), 
\\
\label{eq:tutu2}
\Var \left[ \left( D^{1,\eta}_{\rho=1} \right)_{ij} \right] &=& O_N(\eta^2), 
\\
\label{eq:tutu3}
\Var \left[ \left( D^{2,\eta}_{\rho_1=\rho_2=1} \right)_{ij} \right] &=& O_N(\eta^3),
\end{eqnarray}
where $C^0_N$ is a positive constant.
\end{lemma} 

In practice, we would not necessarily work with $\rho=1$, but with the optimal parameter $\rho^\star$. A direct consequence of~\eqref{eq:tutu2} is of course that
$$
\Var \left[ \left( D^{1,\eta}_{\rho^\star} \right)_{ij} \right] 
=
\inf_\rho \Var \left[ \left( D^{1,\eta}_\rho \right)_{ij} \right]
= O_N(\eta^2).
$$ 

\begin{remark}
Even though the variance of $D^{1,\eta}_{\rho^\star}$ is much smaller than that of $A^\star_{\eta,N}$, we will see in Section~\ref{sec:theo_dD3} below that, in the regime $\eta \ll 1$, the weakly stochastic approximation described in Section~\ref{sec:ALB_result} is even more efficient. 
\end{remark}

\begin{proof}
We infer from~\eqref{eq:toto} and~\eqref{eq:thanks_PBC} that, for any function $\varphi$,
$$
\EE \left[ \varphi \left( A^\star_{\eta,N} \right) \right]
=
\varphi \left( A^\star_{\rm per} \right)
+
\eta |Q_N| \Big( \varphi \left( A^\star_{1,0,N} \right) - \varphi \left( A^\star_{\rm per} \right) \Big)
+
O_N(\eta^2).
$$
Taking $\varphi(M) = M_{ij}$ and $\varphi(M) = M_{ij}^2$, we obtain~\eqref{eq:tutu1}.

\medskip

We next turn to proving~\eqref{eq:tutu2}. For any function $\varphi$, we write, using~\eqref{eq:control-1} and~\eqref{eq:marg1}, that
\begin{eqnarray*}
&& \EE \left[ \varphi \left( D^{1,\eta}_{\rho=1} \right) \right] 
\\
&=& 
(1-\eta)^{|Q_N|} \varphi \left( A^{\star}_{\rm per} + \eta \overline{A}^N_1 \right) 
\\
&& \qquad
+ \sum_{k \in \II_N} \eta (1-\eta)^{|Q_N|-1} \varphi \left( A^\star_{1,k,N} - \overline{A}^{k,N}_{\rm 1 \, def} + \eta \overline{A}^N_1  \right) + O_N(\eta^2)
\\
&=& 
(1-\eta)^{|Q_N|} \varphi \left( A^{\star}_{\rm per} + \eta \overline{A}^N_1 \right) + \sum_{k \in \II_N} \eta (1-\eta)^{|Q_N|-1} \varphi \left( A^\star_{\rm per} + \eta \overline{A}^N_1  \right) + O_N(\eta^2)
\\
&=& 
\varphi \left( A^{\star}_{\rm per} + \eta \overline{A}^N_1 \right) + O_N(\eta^2).
\end{eqnarray*}
Taking $\varphi(M) = M_{ij}$ and $\varphi(M) = M_{ij}^2$, we obtain~\eqref{eq:tutu2}. The proof of~\eqref{eq:tutu3} follows the same lines.
\end{proof}

\subsubsection{Comparison to a weakly stochastic approach}
\label{sec:theo_dD3}

In the regime $\eta \ll 1$, we have three approaches at our disposal to estimate $\EE \left[ A^\star_{\eta,N} \right]$: the standard Monte Carlo approach, the control variate approach, and the weakly stochastic approach described in Section~\ref{sec:ALB_result}. We compare here their efficiency. Let ${\cal C}_N$ be the cost to solve a single corrector problem on $Q_N$.

The standard Monte Carlo approach amounts to writing
$$
\EE \left[ A^\star_{\eta,N} \right] \approx \frac{1}{M} \sum_{m=1}^M A^{\star,m}_{\eta,N}(\omega).
$$
In the above approximation, the error on the entry $ij$ is controlled by $\dis \sqrt{ \Var \left[ \left( A^\star_{\eta,N} \right)_{ij} \right] / M}$. In view of~\eqref{eq:tutu1}, it is thus of the order of $\sqrt{ \eta/M }$. The cost is $M \, {\cal C}_N$.

\medskip

The control variate approach (say using the first order surrogate model) amounts to writing
$$
\EE \left[ A^\star_{\eta,N} \right] \approx \frac{1}{M} \sum_{m=1}^M D^{1,\eta,m}_\rho(\omega),
$$
where $D^{1,\eta}_\rho(\omega)$ is defined by~\eqref{eq:control-1}. The error is of the order of $\sqrt{\eta^2/M}$ in view of~\eqref{eq:tutu2}. The cost is that of solving $M$ corrector problems and that of determining $\overline{A}^{0,N}_{\rm 1 \, def}$, namely $(1+M) \, {\cal C}_N$.

\medskip

Using the same kind of information as in the above control variate approach, the weakly stochastic approximation~\eqref{eq:dl-weakly-random} reads
$$
\EE \left[ A^\star_{\eta,N} \right] \approx A^\star_{\rm per} + \eta \overline{A}^N_1.
$$
The error is of the order of $\eta^2$. The cost is that of determining $\overline{A}^{0,N}_{\rm 1 \, def}$, i.e. ${\cal C}_N$.

\medskip

Obviously, the control variate approach is always more efficient than the Monte Carlo approach. However, to reach the same accuracy as the weakly stochastic approach, one would need to take $M=\eta^{-2}$ realizations, leading to a cost much larger than with the weakly stochastic approach. The same observation holds when using the control variate approach using the second order surrogate model. Therefore, in the regime $\eta \ll 1$, the weakly stochastic approach~\eqref{eq:dl-weakly-random} is the most efficient one.  

\section{Numerical results}
\label{sec:num} 

We consider the so-called random checkerboard case, in dimension $d=2$ (see Fig.~\ref{fig:checkerboard}). It falls into the framework~\eqref{eq:weak-random}--\eqref{eq:random-form1}--\eqref{eq:ber} with
\begin{equation}
\label{eq:def_alpha_beta}
A_{\rm per}(x) = \alpha \Id_2 \quad \text{and} \quad C_{\rm per}(x) = \beta \Id_2.
\end{equation}
In what follows, we choose $\alpha = 3$ and $\beta = 23$ (in Section~\ref{sec:num_low}) or $\beta = 103$ (in Section~\ref{sec:num_high}). All variances are estimated on the basis of $M=100$ independent realizations.

\begin{figure}[htbp]
\centerline{
\includegraphics[width=5cm,angle=90]{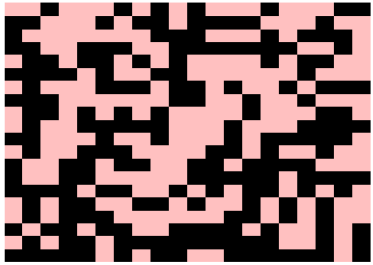}
}
\caption{A typical realization of the checkerboard test-case with $\eta=1/2$.
\label{fig:checkerboard}
}
\end{figure}

\subsection{Low contrast test-case}
\label{sec:num_low}

We choose here $(\alpha,\beta) = (3,23)$. The motivation for this choice is that we already considered this test-case in~\cite{mprf,banff,cedya} when introducing an antithetic variable approach. We are thus in position to compare the results obtained here with our previous results. 

On Fig.~\ref{fig:homogenized_wrt_eta_low}, we plot as a function of $\eta \in (0,1)$ three quantities:
\begin{itemize}
\item the first entry of the matrix $\EE \left[ A^\star_{\eta,N} \right]$ (obtained in practice by an expensive Monte Carlo estimation);
\item the weakly stochastic approximation~\eqref{eq:dl-weakly-random}, which is an approximation of $\EE \left[ A^\star_{\eta,N} \right]$ with an error of the order of $O_N(\eta^3)$;
\item the weakly stochastic approximation obtained in the regime $(1-\eta) \ll 1$, which is an approximation of $\EE \left[ A^\star_{\eta,N} \right]$ with an error of the order of $O_N\big((1-\eta)^3\big)$.
\end{itemize}
In all cases, we work with $N=10$, and the following observations are also valid for larger values of $N$. We see on Fig.~\ref{fig:homogenized_wrt_eta_low} that, when $\eta \leq 0.4$, the deterministic expansion~\eqref{eq:dl-weakly-random} is a very accurate approximation of $\EE \left[ \left( A^\star_{\eta,N} \right)_{11} \right]$. This approximation is inexpensive to compute. The same observation holds in the regime $\eta \geq 0.7$, where the deterministic expansion around $\eta=1$ provides a satisfying approximation. However, we note that none of the two weakly stochastic expansions are accurate when $0.4 \leq \eta \leq 0.7$. In that regime, one has to compute $\EE \left[ \left( A^\star_{\eta,N} \right)_{11} \right]$ by considering several realizations of~\eqref{eq:A-N}--\eqref{eq:correcteur-random-N}. In that regime, considering a variance reduction approach is useful. 

\begin{figure}[htbp]
\centerline{
\includegraphics[width=8cm]{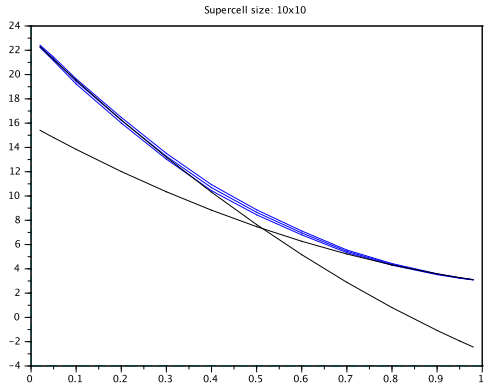}
}
\caption{$\EE \left[ \left( A^\star_{\eta,N} \right)_{11} \right]$ as a function of $\eta$, for $N=10$. Black curves: weakly stochastic approximations. Blue curve: Monte Carlo standard estimator. 
\label{fig:homogenized_wrt_eta_low}
}
\end{figure}

\medskip

In the regime we have identified, we show on Fig.~\ref{fig:var_red_wrt_eta_low_mid} the ratios of variance 
\begin{equation}
\label{eq:ratio_eff}
R_{\eta,N} = \frac{\Var\Big(\left[ A^\star_{\eta,N} \right]_{11}\Big)}{\Var(D)},
\end{equation}
where $D$ is either the first-order controlled variable $D_\rho^{1,\eta}(\omega)$ defined by~\eqref{eq:control-1}, or the second-order controlled variable $D_{\rho_1,\rho_2}^{2,\eta}(\omega)$ defined by~\eqref{eq:control-2}, or the controlled variable $D_{\rho_1,\rho_2,\rho_3}^{3,\eta}(\omega)$ defined by~\eqref{eq:control-3}. The parameter $\rho$ (resp. $(\rho_1,\rho_2)$ and $(\rho_1,\rho_2,\rho_3)$) is chosen to minimize the variance of the estimator. In this section, we exactly compute (up to finite element errors) the quantities $\overline{A}^{k,l,N}_{\rm 2 \, def}$ needed to build the controlled variables~\eqref{eq:control-2} and~\eqref{eq:control-3}. In Section~\ref{sec:num_RB} below, we approximate them using a Reduced Basis approach. We postpone until that section the discussion on computational costs and only focus here on accuracy.

\begin{remark}
The second-order controlled variable $D_{\rho_1,\rho_2}^{2,\eta}(\omega)$ defined by~\eqref{eq:control-2} is built by considering $A_{\rm per}$ as the reference. One could alternatively build a second-order controlled variable considering $C_{\rm per}$ as the reference. Numerical results obtained with such a controlled variable are similar to those obtained with $D_{\rho_1,\rho_2}^{2,\eta}(\omega)$ (results not shown). 
\end{remark}

\begin{figure}[htbp]
\centerline{
\includegraphics[width=8cm]{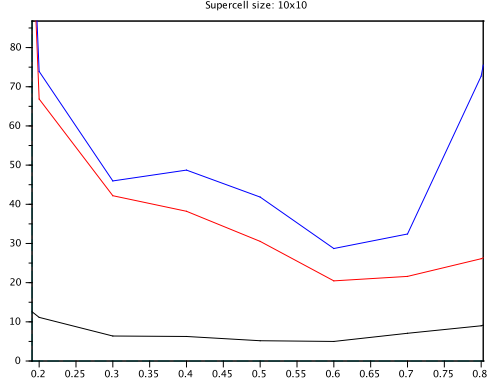}
}
\caption{Ratio $R_{\eta,N}$ defined by~\eqref{eq:ratio_eff} as a function of $\eta$ ($N=10$). Black curve: controlled variable $D_\rho^{1,\eta}(\omega)$. Red curve: controlled variable $D_{\rho_1,\rho_2}^{2,\eta}(\omega)$. Blue curve: controlled variable $D_{\rho_1,\rho_2,\rho_3}^{3,\eta}(\omega)$.
\label{fig:var_red_wrt_eta_low_mid}
}
\end{figure}

\medskip

We observe on Fig.~\ref{fig:var_red_wrt_eta_low_mid} that, for $\eta = 1/2$, the approach using the first-order controlled variable~\eqref{eq:control-1} provides a variance reduction ratio~\eqref{eq:ratio_eff} close to 6. This gain is close to the gain obtained using an antithetic variable approach (see~\cite[Table 2]{cedya}). In contrast, when using the controlled variable~\eqref{eq:control-3} taking into account first order and second order corrections with respect to both the cases $\eta=0$ and $\eta=1$, we obtain a gain close to 40. 

\medskip

We now monitor how the gain depends on the size of the domain $Q_N$. To that aim, we show on Table~\ref{tab:var_red_wrt_eta_low} the ratio~\eqref{eq:ratio_eff} as a function of $N$, for $\eta=1/2$. We observe that the gain is essentially independent of $N$. 

\begin{table}[htdp]
\begin{center}
\begin{tabular}{|c|c|c|c|c|c|}
\hline
& $N=6$ & $N=10$ & $N=20$ & $N=30$ & $N=50$ \\
\hline
First order & 7.57 & 5.18 & 6.55 & 8.51 & 7.34 \\
\hline
Second order & 35.9 & 41.8 & 37.6 & 35.6 & 40.4 \\
\hline
\end{tabular}
\end{center}
\caption{Ratio $R_{\eta,N}$ defined by~\eqref{eq:ratio_eff} as a function of $N$ ($\eta=1/2$). First order: controlled variable $D_\rho^{1,\eta}(\omega)$. Second order: controlled variable $D_{\rho_1,\rho_2,\rho_3}^{3,\eta}(\omega)$.
\label{tab:var_red_wrt_eta_low}}
\end{table}

\begin{remark}
In the one-dimensional case, we have shown that the variance ratio is proportional to $N$ or $N^2$ (see Propositions~\ref{prop:theo_1D} and~\ref{prop:theo_1D_suite}). In the two-dimensional case, we do not observe such an excellent behavior for our approach. The gain rather seems to be independent of $N$ (see also Fig.~\ref{fig:var_red_RB}). Nevertheless, the variance ratio is significantly higher than 1, making the approach definitely superior to the standard Monte Carlo approach.
\end{remark}

\subsection{High contrast test-case}
\label{sec:num_high}

We now turn to a test-case with a larger contrast and set $(\alpha,\beta)= (3,103)$ in~\eqref{eq:def_alpha_beta}. On Fig.~\ref{fig:homogenized_wrt_eta_high}, we plot as a function of $\eta \in (0,1)$ the same three quantities as on Fig.~\ref{fig:homogenized_wrt_eta_low} (again with $N=10$). We again see that, when $0.3 \leq \eta \leq 0.7$, none of the two weakly stochastic expansions are accurate. This is the regime we focus on.

We also show on Fig.~\ref{fig:homogenized_wrt_eta_high} the ratios of variance~\eqref{eq:ratio_eff} for the same three control variate approaches as on Fig.~\ref{fig:var_red_wrt_eta_low_mid}. We observe that, for $\eta = 1/2$, the approach using the controlled variable~\eqref{eq:control-3} provides a gain close to 6.7. This gain is smaller than in the case of Section~\ref{sec:num_low} (the contrast is now larger), but still significant. As in the low-contrast test-case, the gain is essentially independent of $N$, as shown in Table~\ref{tab:var_red_wrt_eta_high}.

\begin{figure}[htbp]
\begin{center}
\includegraphics[width=6.2cm]{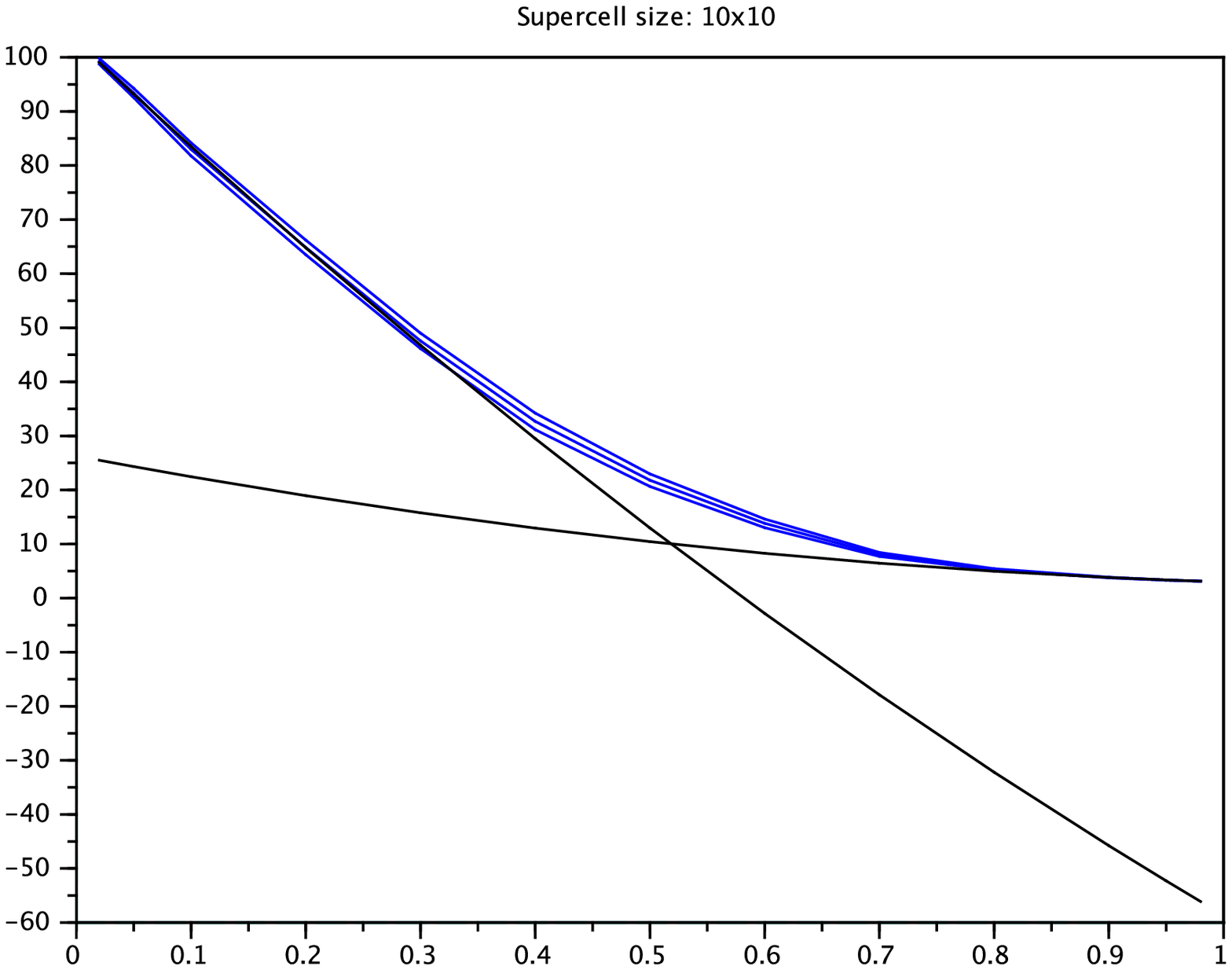}
\includegraphics[width=6.2cm]{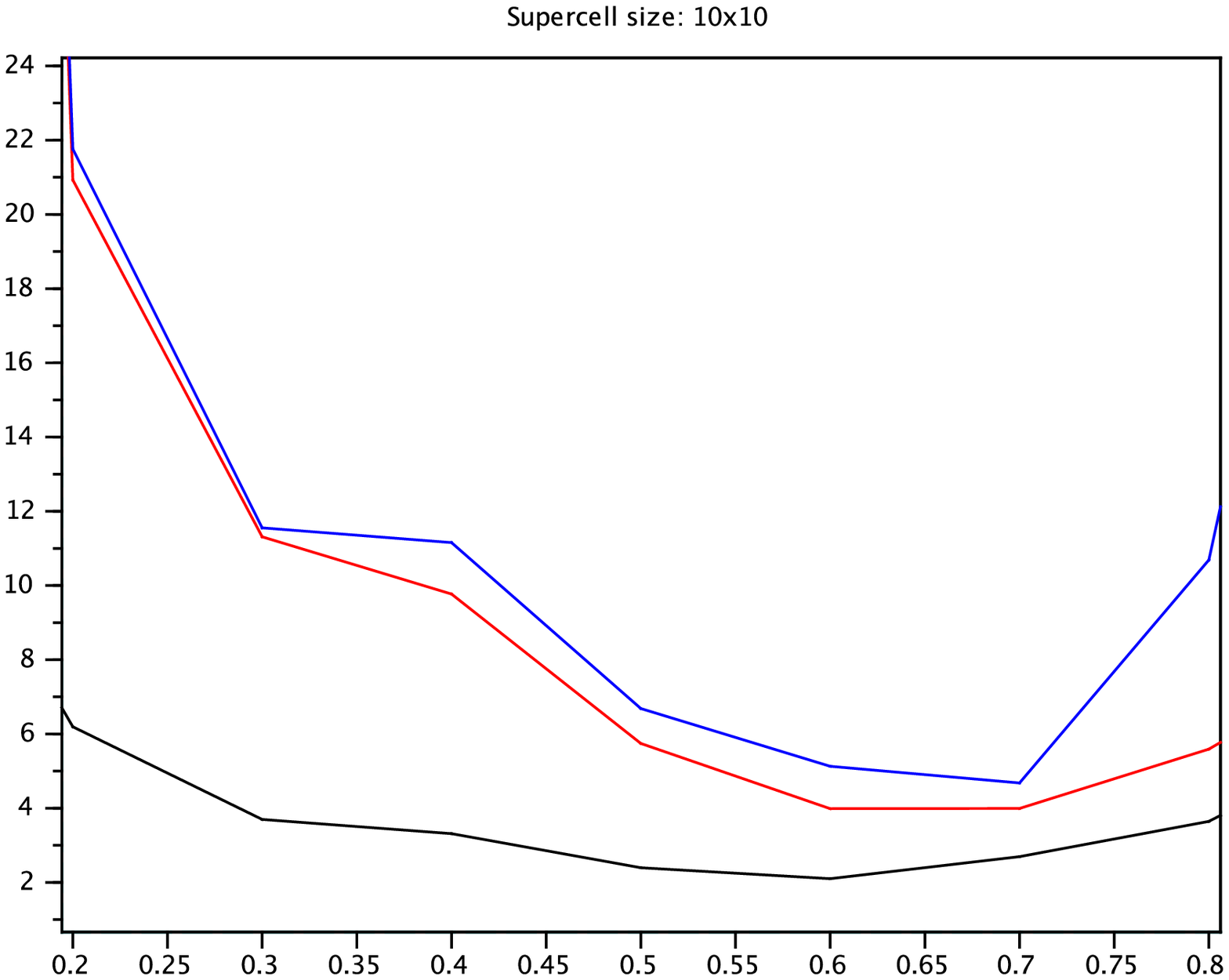}
\end{center}
\caption{Left: $\dis \EE \left[ \left( A^\star_{\eta,N} \right)_{11} \right]$ as a function of $\eta$, for $N=10$. Blue curve: standard Monte Carlo estimator. Black curves: weakly stochastic approximations. Right: Ratio $R_{\eta,N}$ defined by~\eqref{eq:ratio_eff} as a function of $\eta$ ($N=10$). Black curve: controlled variable $D_\rho^{1,\eta}(\omega)$. Red curve: controlled variable $D_{\rho_1,\rho_2}^{2,\eta}(\omega)$. Blue curve: controlled variable $D_{\rho_1,\rho_2,\rho_3}^{3,\eta}(\omega)$.
\label{fig:homogenized_wrt_eta_high}
}
\end{figure}

\begin{table}[htbp]
\begin{center}
\begin{tabular}{|c|c|c|c|}
\hline
& $N=10$ & $N=30$ & $N=50$ \\
\hline
First order & 2.40 & 3.62 &  3.87 \\
\hline
Second order & 6.69 & 6.32 & 5.82 \\
\hline
\end{tabular}
\end{center}
\caption{Ratio $R_{\eta,N}$ defined by~\eqref{eq:ratio_eff} as a function of $N$ ($\eta=1/2$). First order: controlled variable $D_\rho^{1,\eta}(\omega)$. Second order: controlled variable $D_{\rho_1,\rho_2,\rho_3}^{3,\eta}(\omega)$.
\label{tab:var_red_wrt_eta_high}}
\end{table}

\subsection{Using a Reduced Basis (RB) approach}
\label{sec:num_RB}

In Sections~\ref{sec:num_low} and~\ref{sec:num_high}, we have used the second-order surrogate model~\eqref{eq:control-3}, which takes into account the contributions from pairs of defects located at any site $k$ and $l$, namely $A^\star_{2,k,l,N}$ defined by~\eqref{eq:A-2N} and $C^\star_{2,k,l,N}$ defined by~\eqref{eq:D-2N}. These quantities are deterministic, and computed beforehand. However, in practice, computing these quantities is expensive, because we have to consider all possible configurations of pairs of defects.

This high computational cost can be decreased by using the Reduced Basis (RB) approach proposed in~\cite{clb_thomines}. This approach amounts to solving the one-defect problem~\eqref{eq:correcteur-1N}, and {\em a few} two-defects problems~\eqref{eq:correcteur-2N}, for $k=0$ and $l$ in some set ${\cal N}_N \subset \II_N \setminus \{ 0 \}$ (in practice, we solve~\eqref{eq:correcteur-2N} for some $l$ close to $k$). Then, it turns out that the solutions to the other two-defects problems, i.e. $w_p^{2,k,l,N}$ for $k=0$ and $l \notin {\cal N}_N$, can be well-approximated on the basis of $w_p^{1,0,N}$ and $\left\{ w_p^{2,k,l,N} \right\}_{k=0, \, l \in {\cal N}_N}$. 

In the sequel, we consider the low-contrast test-case (i.e. $(\alpha,\beta)= (3,23)$ in~\eqref{eq:def_alpha_beta}), set $\eta=1/2$, and use this RB approach in order to decrease the offline cost of our control variate approach.

\subsubsection{Robutness with respect to the RB basis set}

First, we evaluate the robustness of the gain in variance when we approximate the quantities $A^\star_{2,k,l,N}$ and $C^\star_{2,k,l,N}$ by the above RB approach, in contrast to computing them exactly (i.e., up to a small Finite Element error). To do so, we fix $N$ and monitor the variance ratio for the sets ${\cal N}_N$ shown on Fig.~\ref{fig:set_Nn}. Results are given in Table~\ref{tab:RB_N}. We see that the gain in variance is independent of the set ${\cal N}_N$: we can use the RB approach with a very small set of configurations for which the correctors $w_p^{2,k,l,N}$ are exactly computed (thereby dramatically decreasing the offline computational cost), and still retain an excellent variance reduction. 

\begin{figure}[htbp]
\begin{center}
  \def\Lnet{0.7}    
  \begin{tikzpicture}                               
    \draw[style=help lines] (-1.5*\Lnet,-1.5*\Lnet) grid[step=\Lnet] (6.5*\Lnet,6.5*\Lnet);
          \draw [fill=red,ultra thick] (1*\Lnet,0*\Lnet) rectangle (2*\Lnet,1*\Lnet);
          \draw [fill=red,ultra thick] (0*\Lnet,1*\Lnet) rectangle (1*\Lnet,2*\Lnet);
          \draw [fill=red,ultra thick] (0*\Lnet,3*\Lnet) rectangle (1*\Lnet,4*\Lnet);
          \draw [fill=red,ultra thick] (3*\Lnet,0*\Lnet) rectangle (4*\Lnet,1*\Lnet);
          \draw [fill=red,ultra thick] (4*\Lnet,3*\Lnet) rectangle (5*\Lnet,4*\Lnet);
          \draw [fill=red,ultra thick] (3*\Lnet,4*\Lnet) rectangle (4*\Lnet,5*\Lnet);
          \draw [fill=red,ultra thick] (4*\Lnet,1*\Lnet) rectangle (5*\Lnet,2*\Lnet);
          \draw [fill=red,ultra thick] (1*\Lnet,4*\Lnet) rectangle (2*\Lnet,5*\Lnet);
          
          \draw [fill=blue,ultra thick] (4*\Lnet,2*\Lnet) rectangle (5*\Lnet,3*\Lnet);
          \draw [fill=blue,ultra thick] (2*\Lnet,4*\Lnet) rectangle (3*\Lnet,5*\Lnet);
          \draw [fill=blue,ultra thick] (0*\Lnet,2*\Lnet) rectangle (1*\Lnet,3*\Lnet);
          \draw [fill=blue,ultra thick] (2*\Lnet,0*\Lnet) rectangle (3*\Lnet,1*\Lnet);

          \draw [fill=purple,ultra thick] (2*\Lnet,1*\Lnet) rectangle (3*\Lnet,2*\Lnet);
          \draw [fill=purple,ultra thick] (1*\Lnet,2*\Lnet) rectangle (2*\Lnet,3*\Lnet);
          \draw [fill=purple,ultra thick] (3*\Lnet,2*\Lnet) rectangle (4*\Lnet,3*\Lnet);
          \draw [fill=purple,ultra thick] (2*\Lnet,3*\Lnet) rectangle (3*\Lnet,4*\Lnet);

          \draw [fill=orange,ultra thick] (3*\Lnet,3*\Lnet) rectangle (4*\Lnet,4*\Lnet);
          \draw [fill=orange,ultra thick] (1*\Lnet,1*\Lnet) rectangle (2*\Lnet,2*\Lnet);
          \draw [fill=orange,ultra thick] (1*\Lnet,3*\Lnet) rectangle (2*\Lnet,4*\Lnet);
          \draw [fill=orange,ultra thick] (3*\Lnet,1*\Lnet) rectangle (4*\Lnet,2*\Lnet);
\end{tikzpicture}
\qquad 
  \def\Lnet{0.7}    
  \begin{tikzpicture}                               
    \draw[style=help lines] (-1.5*\Lnet,-1.5*\Lnet) grid[step=\Lnet] (6.5*\Lnet,6.5*\Lnet);
          \draw [fill=blue,ultra thick] (4*\Lnet,2*\Lnet) rectangle (5*\Lnet,3*\Lnet);
          \draw [fill=blue,ultra thick] (2*\Lnet,4*\Lnet) rectangle (3*\Lnet,5*\Lnet);
          \draw [fill=blue,ultra thick] (0*\Lnet,2*\Lnet) rectangle (1*\Lnet,3*\Lnet);
          \draw [fill=blue,ultra thick] (2*\Lnet,0*\Lnet) rectangle (3*\Lnet,1*\Lnet);

          \draw [fill=purple,ultra thick] (2*\Lnet,1*\Lnet) rectangle (3*\Lnet,2*\Lnet);
          \draw [fill=purple,ultra thick] (1*\Lnet,2*\Lnet) rectangle (2*\Lnet,3*\Lnet);
          \draw [fill=purple,ultra thick] (3*\Lnet,2*\Lnet) rectangle (4*\Lnet,3*\Lnet);
          \draw [fill=purple,ultra thick] (2*\Lnet,3*\Lnet) rectangle (3*\Lnet,4*\Lnet);

          \draw [fill=orange,ultra thick] (3*\Lnet,3*\Lnet) rectangle (4*\Lnet,4*\Lnet);
          \draw [fill=orange,ultra thick] (1*\Lnet,1*\Lnet) rectangle (2*\Lnet,2*\Lnet);
          \draw [fill=orange,ultra thick] (1*\Lnet,3*\Lnet) rectangle (2*\Lnet,4*\Lnet);
          \draw [fill=orange,ultra thick] (3*\Lnet,1*\Lnet) rectangle (4*\Lnet,2*\Lnet);
\end{tikzpicture}
\bigskip
\bigskip
  \def\Lnet{0.7}    
  \begin{tikzpicture}                               
    \draw[style=help lines] (-1.5*\Lnet,-1.5*\Lnet) grid[step=\Lnet] (6.5*\Lnet,6.5*\Lnet);
          \draw [fill=purple,ultra thick] (2*\Lnet,1*\Lnet) rectangle (3*\Lnet,2*\Lnet);
          \draw [fill=purple,ultra thick] (1*\Lnet,2*\Lnet) rectangle (2*\Lnet,3*\Lnet);
          \draw [fill=purple,ultra thick] (3*\Lnet,2*\Lnet) rectangle (4*\Lnet,3*\Lnet);
          \draw [fill=purple,ultra thick] (2*\Lnet,3*\Lnet) rectangle (3*\Lnet,4*\Lnet);

          \draw [fill=orange,ultra thick] (3*\Lnet,3*\Lnet) rectangle (4*\Lnet,4*\Lnet);
          \draw [fill=orange,ultra thick] (1*\Lnet,1*\Lnet) rectangle (2*\Lnet,2*\Lnet);
          \draw [fill=orange,ultra thick] (1*\Lnet,3*\Lnet) rectangle (2*\Lnet,4*\Lnet);
          \draw [fill=orange,ultra thick] (3*\Lnet,1*\Lnet) rectangle (4*\Lnet,2*\Lnet);
\end{tikzpicture}
\qquad
  \def\Lnet{0.7}    
  \begin{tikzpicture}                               
    \draw[style=help lines] (-1.5*\Lnet,-1.5*\Lnet) grid[step=\Lnet] (6.5*\Lnet,6.5*\Lnet);
           \draw [fill=purple,ultra thick] (2*\Lnet,1*\Lnet) rectangle (3*\Lnet,2*\Lnet);
          \draw [fill=purple,ultra thick] (1*\Lnet,2*\Lnet) rectangle (2*\Lnet,3*\Lnet);
          \draw [fill=purple,ultra thick] (3*\Lnet,2*\Lnet) rectangle (4*\Lnet,3*\Lnet);
          \draw [fill=purple,ultra thick] (2*\Lnet,3*\Lnet) rectangle (3*\Lnet,4*\Lnet);

 \end{tikzpicture}
\end{center}
\caption{Sets ${\cal N}_N$ of position of second defect that we consider to build the RB basis set (the first defect is always in the central white cell). 
Top left: $\text{Card } {\cal N}_N = 20$.
Top right: $\text{Card } {\cal N}_N = 12$.
Bottom left: $\text{Card } {\cal N}_N = 8$.
Bottom right: $\text{Card } {\cal N}_N = 4$.
\label{fig:set_Nn}
}
\end{figure}
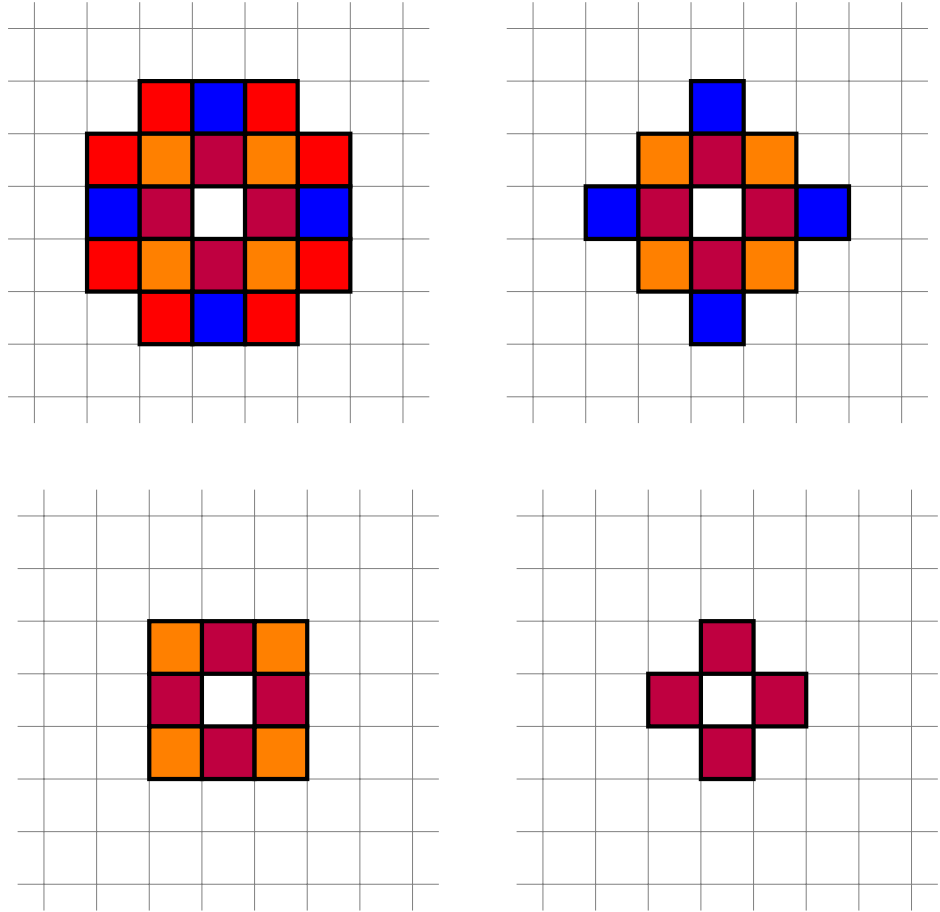

\begin{table}[htbp]
\begin{center}
\begin{tabular}{|c|c|c|}
\hline
& $N=6$ & $N=20$ \\
\hline
${\cal N}_N = \II_N \setminus \{ 0 \}$ & 35.9 & 37.6 \\
\hline
$\text{Card } {\cal N}_N = 20$ & 36.1 & 37.6 \\
\hline
$\text{Card } {\cal N}_N = 12$ & 35.7 & 37.0 \\
\hline
$\text{Card } {\cal N}_N = 8$ & 36.6 & 36.5 \\
\hline
$\text{Card } {\cal N}_N = 4$ & 36.6 & 37.6 \\
\hline
\end{tabular}
\end{center}
\caption{Ratio $R_{\eta,N}$ defined by~\eqref{eq:ratio_eff} for two values of $N$ ($\eta=1/2$), using the second order model $D_{\rho_1,\rho_2,\rho_3}^{3,\eta}(\omega)$ defined by~\eqref{eq:control-3}. The first line corresponds to the reference computation of $A^\star_{2,k,l,N}$ and $C^\star_{2,k,l,N}$. The subsequent lines correspond to using a RB approach to compute $A^\star_{2,k,l,N}$ and $C^\star_{2,k,l,N}$, with a decreasing set ${\cal N}_N$.
\label{tab:RB_N}}
\end{table}

\medskip

Following the above idea, we have also tested the approach when we set $\overline{A}^{k,l,N}_{\rm 2 \, def} = \overline{C}^{k,l,N}_{\rm 2 \, def} = \Id$ in~\eqref{eq:Y2} and~\eqref{eq:Y3} for any $k \neq l$ (which amounts to setting $A^\star_{2,k,l,N} = \Id + 2 A^\star_{1,0,N} - A^\star_{\rm per}$, see~\eqref{eq:marg2}). We do not expect (and this is indeed the case) to obtain good results. The controlled variable reads
\begin{multline}
\label{eq:control-3-zozo}
D_{\rho_1,\rho_2,\rho_3}^{3,\eta,{\rm approx}}(\omega) 
= 
A^\star_{\eta,N}(\omega) - \rho_1 \left( A^{\eta,N}_1 (\omega) - \eta \overline{A}^N_1 \right) 
\\
- \frac{\rho_2}{2} \sum_{k \neq l \in \II_N} \Big( B^\eta_k (\omega) B^\eta_l (\omega) - \EE \left[ B^\eta_k B^\eta_l \right] \Big) 
\\
- \frac{\rho_3}{2} \sum_{k \neq l \in \II_N} \Big( (1-B^\eta_k(\omega)) (1-B^\eta_l(\omega)) - \EE \left[ (1-B^\eta_k) (1-B^\eta_l) \right] \Big)
\end{multline}
instead of~\eqref{eq:control-3}. Computing the second order surrogate model is then extremely cheap, and as expensive as computing the first order surrogate model: one only has to solve the one-defect problem~\eqref{eq:correcteur-1N}. In that case, for $N=20$ and $\eta=1/2$, the variance ratio is equal to 6.96, which is extremely close to the variance ratio obtained by simply using the first order model (see Table~\ref{tab:var_red_wrt_eta_low}), which is equal to 6.55. Considering the last two lines in~\eqref{eq:control-3-zozo} therefore does not improve the efficiency.

\medskip

The above results show that it is not needed to compute with a high accuracy the quantities $A^\star_{2,k,l,N}$ and $C^\star_{2,k,l,N}$ to obtain a significant variance reduction. Using a RB approach with a very small set ${\cal N}_N$ is sufficient and the gain (in terms of variance reduction) is essentially the same as that if $A^\star_{2,k,l,N}$ and $C^\star_{2,k,l,N}$ are exactly computed. However, even though the approach is quite flexible, it still requires approximations of $A^\star_{2,k,l,N}$ and $C^\star_{2,k,l,N}$ with a reasonable accuracy. Otherwise, the efficiency significantly drops down, as shown by our last test. 

\subsubsection{Results as a function of $N$}

We now fix the RB basis set corresponding to $\text{Card } {\cal N}_N = 12$ on Fig.~\ref{fig:set_Nn}, and compare the Monte Carlo results with our control variate results, using the controlled variable~\eqref{eq:control-3}. To evaluate the Monte Carlo estimator
$$
I^{\rm MC}_M = \frac{1}{M} \sum_{m=1}^M A^{\star,m}_{\eta,N}(\omega),
$$
we need to solve $M$ corrector problems. In contrast, to evaluate the Control Variate estimator
$$
I^{\rm CV}_M := \frac{1}{M} \sum_{m=1}^M D_{\rho_1,\rho_2,\rho_3}^{3,\eta,m}(\omega), 
$$
we need to solve first the problem~\eqref{eq:correcteur-1N} and the problems~\eqref{eq:correcteur-2N} for $k=0$ and $l \in {\cal N}_N$, and second $M$ corrector problems. Let ${\cal C}_N$ be the cost to solve a single corrector problem on $Q_N$. Then the Monte Carlo cost is $M \, {\cal C}_N$, the Control Variate offline cost is $(1+{\cal N}_N) \, {\cal C}_N = 13 {\cal C}_N$, and its online cost is $M \, {\cal C}_N$. In the sequel, we work with $M=100$, therefore the Control Variate cost is just 13\% higher than the Monte Carlo cost.

\medskip

First, we plot on Fig.~\ref{fig:homogenized_RB} the confidence intervals obtained for the Monte Carlo approach and the Control Variate approach based on~\eqref{eq:control-3}. The latter confidence interval width is dramatically smaller than the former. 

\begin{figure}[htbp]
\centerline{
\includegraphics[width=10cm]{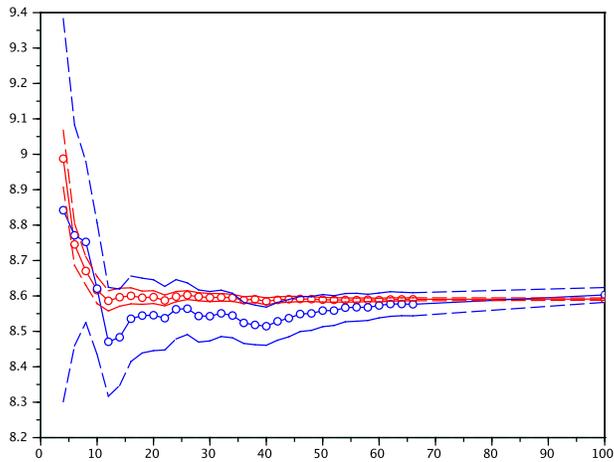}
}
\caption{Estimation of $\EE \left( \left[ A^\star_{\eta,N} \right]_{11} \right)$ as a function of $N$. Blue: standard Monte-Carlo estimator. Red: Control Variate estimator based on~\eqref{eq:control-3}. In both cases, estimators are built using $M=100$ i.i.d. realizations.
\label{fig:homogenized_RB}
}
\end{figure}

\medskip

We next show on Fig.~\ref{fig:var_red_RB} the variance ratios~\eqref{eq:ratio_eff}. They somewhat vary with $N$. Recall that these ratios are computed on the basis of $M=100$ i.i.d. realizations. From one set of i.i.d. realizations to another, results may slightly vary, although qualitative conclusions remain alike. For the first order method based on~\eqref{eq:control-1}, the variance ratio is between 5 and 10, whereas it is around 30 or more for the second order method based on~\eqref{eq:control-3}.

\begin{figure}[htbp]
\centerline{
\includegraphics[width=8cm]{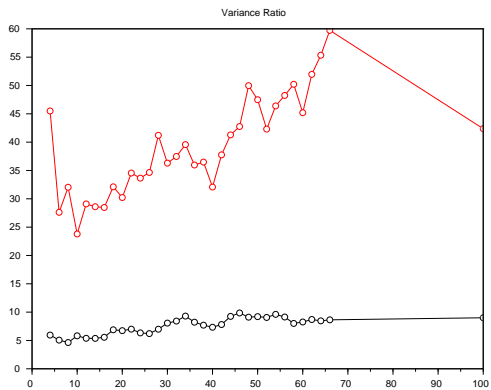}
}
\caption{Variance ratio~\eqref{eq:ratio_eff} as a function of $N$. Black curve: using the first order controlled variable~\eqref{eq:control-1}. Red curve: using the second order controlled variable~\eqref{eq:control-3}. We have considered all values $N \in \{4, 6,\dots, 66 \}$ as well as $N=100$.
\label{fig:var_red_RB}}
\end{figure}

\medskip

We plot on Fig.~\ref{fig:control_magnitude} the optimal values of $\rho_1$, $\rho_2$ and $\rho_3$, solution to~\eqref{eq:system_rho}. None of these parameters is close to 0: all random variables $A^{\eta,N}_1(\omega)$, $A^{\eta,N}_2(\omega)$ and $C^{\eta,N}_2(\omega)$ are useful in~\eqref{eq:control-3} to decrease the variance. 

\begin{figure}[htbp]
\centerline{
\includegraphics[width=8cm]{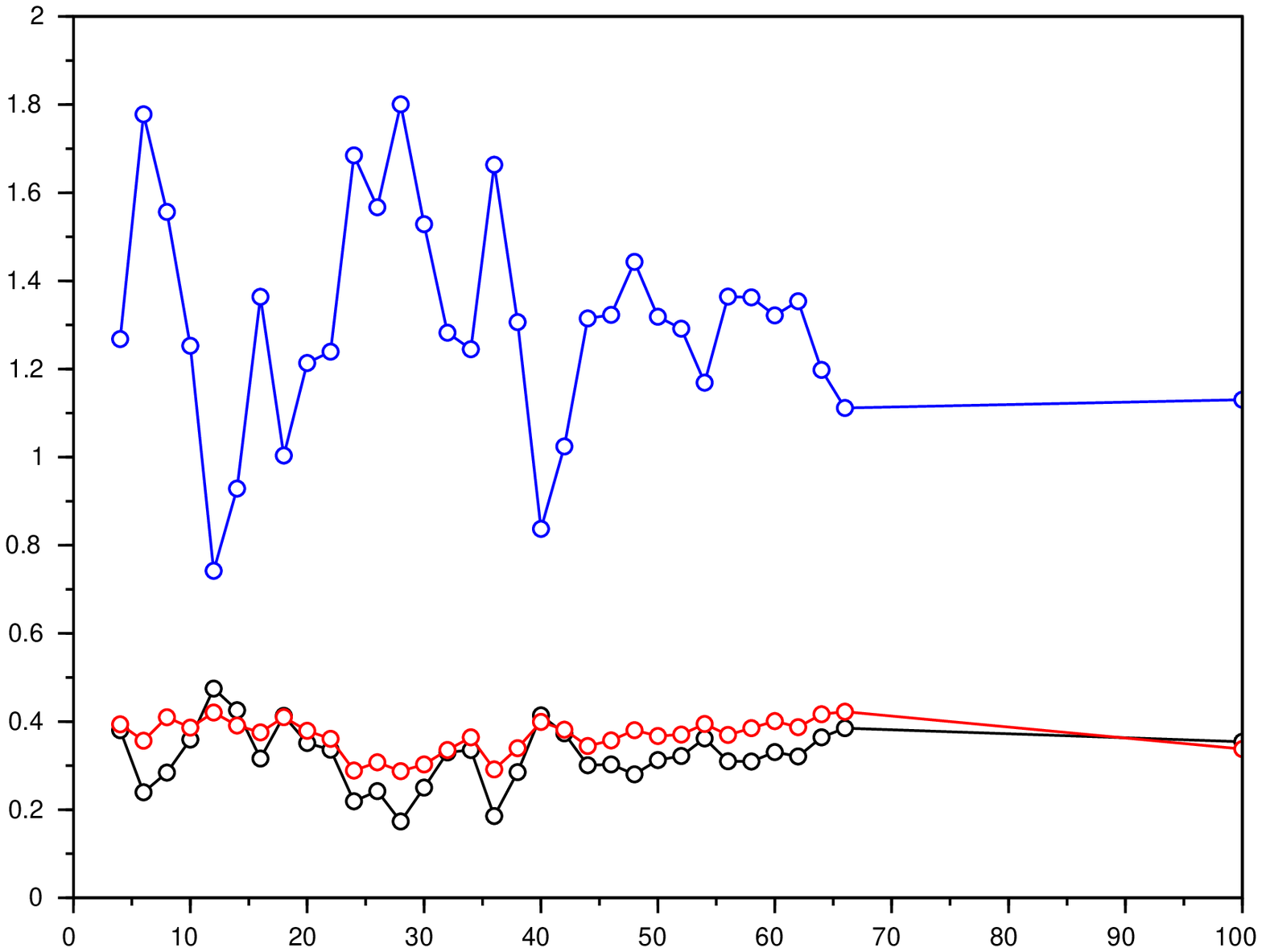}
}
\caption{Optimal values of $\rho_1$ (black), $\rho_2$ (red) and $\rho_3$ (blue) for the controlled variable~\eqref{eq:control-3} as a function of $N$. We have considered all values $N \in \{4, 6,\dots, 66 \}$ as well as $N=100$.
\label{fig:control_magnitude}}
\end{figure}

\medskip

On Fig.~\ref{fig:bias_RB}, we eventually plot the complete errors, that is
\begin{equation}
\label{eq:errors}
e^{\rm MC}_{N,M} = \left| \frac{1}{M} \sum_{m=1}^M A^{\star,m}_{\eta,N}(\omega) - A^\star_\eta \right|,
\quad
e^{\rm CV}_{N,M} = \left| \frac{1}{M} \sum_{m=1}^M D^{3,\eta,m}_{\rho_1,\rho_2,\rho_3}(\omega) - A^\star_\eta \right|,
\end{equation}
where the exact value $A^\star_\eta$ is actually approximated using $M_{\rm ref}$ realizations on a large domain $Q_{N_{\rm ref}}$. These errors are a sum of:
\begin{itemize}
\item the bias error $\EE \left[A^\star_{\eta,N} \right] - A^\star_\eta$,
\item the statistical error, which scales as $\dis \sqrt{ \Var \Big( A^\star_{\eta,N} \Big) / M }$ for the Monte-Carlo approach and $\dis \sqrt{ \Var \Big( D^{3,\eta}_{\rho_1,\rho_2,\rho_3} \Big) / M}$ for the Control Variate approach.
\end{itemize}
When $d \geq 3$, the variance of $A^\star_{\eta,N}$ has been shown to scale as $N^{-d}$ in~\cite[Theorem 1.3 and Proposition 1.4]{nolen}. For homogenization problems set on random {\em lattices}, optimal estimates on the above two errors have been established in~\cite[Theorem 2]{GNO_invent} for any $d \geq 2$: the former scales $N^{-d} (\ln N)^d$ while $\Var \Big( A^\star_{\eta,N} \Big)$ scales as $N^{-d}$. 

In the standard Monte Carlo approach, for large values of $N$, we expect the statistical error to dominate, and thus the error to be of the order of $N^{-d/2}$. This is indeed what we observe on the blue curve of Fig.~\ref{fig:bias_RB}. For the Control Variate approach, we observe that the error decreases as $N^{-d}$ (see red curve of Fig.~\ref{fig:bias_RB}). This is consistent with the fact that, for the values of $N$ we consider, the statistical error has been dramatically decreased and is now smaller than the bias error. 

\begin{figure}[htbp]
\centerline{
\includegraphics[width=8cm]{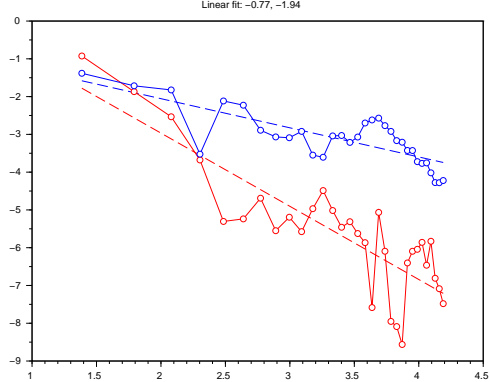}
}
\caption{Errors~\eqref{eq:errors} as a function of $N$ ($M=100$; log-log plot). Blue curve (with slope -0.77): Monte-Carlo approach. Red curve (with slope -1.94): Control Variate approach using~\eqref{eq:control-3}. The reference value has been computed using $N_{\rm ref} = 100$ and $M_{\rm ref}=100$.
\label{fig:bias_RB}}
\end{figure}

\section*{Acknowledgments} 

The work of FL and WM is partially supported by ONR under Grant N00014-12-1-0383 and EOARD under grant FA8655-13-1-3061. WM gratefully acknowledges the support from Labex MMCD (Multi-Scale Modelling \& Experimentation of Materials for Sustainable Construction) under contract ANR-11-LABX-0022. We also wish to thank Claude Le Bris and Xavier Blanc for enlightning discussions.


\begin{thebibliography}{99}

\bibitem{singapour} 
A. Anantharaman, R. Costaouec, C. Le Bris, F. Legoll and F. Thomines, 
{\em Introduction to numerical stochastic homogenization and the related
computational challenges: some recent developments}, W. Bao and Q. Du
eds., Lecture Notes Series, Institute for Mathematical Sciences,
National University of Singapore, vol. 22, 197-272 (2011).

\bibitem{ALB_cras} A. Anantharaman and C. Le Bris, {\em Homog\'en\'eisation d'un mat\'eriau p\'eriodique faiblement perturb\'e al\'eatoirement [Homogenization of a weakly randomly perturbed periodic material]}, C. R. Math. Acad. Sci. Paris, 348(9-10):529-534, 2010.

\bibitem{ALB_ccp} A. Anantharaman and C. Le Bris, {\em Elements of mathematical foundations for numerical approaches for weakly random homogenization problems}, Communications in Computational Physics, 11(4):1103-1143, 2012.

\bibitem{ALB_mms} A. Anantharaman and C. Le Bris, {\em A numerical approach related to defect-type theories for some weakly random problems in homogenization}, SIAM Multiscale Model. Simul., 9(2):513-544, 2011.

\bibitem{blp} A. Bensoussan, J.-L. Lions and G. Papanicolaou, {\bf Asymptotic
 analysis for periodic structures}, Studies in Mathematics and its
 Applications, vol. 5. North-Holland Publishing Co., Amsterdam-New York,
 1978.

\bibitem{banff} X. Blanc, R. Costaouec, C. Le Bris and F. Legoll, 
{\em Variance reduction in stochastic homogenization: the
technique of antithetic variables}, in Numerical Analysis and Multiscale
Computations, B. Engquist, O. Runborg and R. Tsai eds., Lect. Notes
Comput. Sci. Eng., vol. 82, Springer, 47-70 (2012).

\bibitem{mprf} X. Blanc, R. Costaouec, C. Le Bris and F. Legoll,
{\em Variance reduction in stochastic homogenization using antithetic
    variables}, Markov Processes and Related Fields, 18(1):31-66, 2012
(preliminary version available at
{\tt http://cermics.enpc.fr/$\sim$legoll/hdr/FL24.pdf}).

\bibitem{cras-stoc} X. Blanc, C. Le Bris and P.-L. Lions, {\em Une variante
de la th\'eorie de l'homog\'en\'eisation stochastique des op\'erateurs
elliptiques [A variant of stochastic homogenization theory for
elliptic operators]}, C. R. Acad. Sci. S\'erie I, 343(11-12):717-724, 2006. 

\bibitem{jmpa} X. Blanc, C. Le Bris and P.-L. Lions, {\em Stochastic
homogenization and random lattices}, J. Math. Pures Appl., 88(1):34-63, 2007. 

\bibitem{mb-fl} M. Bornert and F. Legoll, in preparation.

\bibitem{bourgeat} A. Bourgeat and A. Piatnitski,
{\em Approximation of effective coefficients in stochastic homogenization},
Ann. I. H. Poincar\'e - PR, 40(2):153-165, 2004.

\bibitem{cd} D. Cioranescu and P. Donato, {\bf An introduction to
   homogenization}, Oxford Lecture Series in Mathematics and its
 Applications, vol. 17. Oxford University Press, New York, 1999.

\bibitem{cedya} R. Costaouec, C. Le Bris and F. Legoll, {\em Variance reduction in stochastic homogenization: proof of concept, using antithetic variables}, Boletin Soc. Esp. Mat. Apl., 50:9-27, 2010.

\bibitem{engquist-souganidis} B. Engquist and P.E. Souganidis, {\em Asymptotic and numerical homogenization}, Acta Numerica, 17:147-190, 2008.

\bibitem{fishman} G.S. Fishman, {\bf Monte Carlo: concepts, algorithms, and applications}, Springer, 1996. 

\bibitem{GNO_invent} A. Gloria, S. Neukamm and F. Otto, {\em Quantification of ergodicity in stochastic homogenization: optimal bounds via spectral gap on Glauber dynamics}, Invent. Math., DOI 10.1007/s00222-014-0518-z, published online in 2014.

\bibitem{jikov} V.V. Jikov, S.M. Kozlov and O.A. Oleinik, {\bf
   Homogenization of differential operators and integral functionals},
   Springer-Verlag, 1994.

\bibitem{krengel} U. Krengel, {\bf Ergodic theorems}, de Gruyter
 Studies in Mathematics, vol. 6, de Gruyter, 1985.

\bibitem{enumath} C. Le Bris, {\em Some numerical approaches for ``weakly'' random homogenization}, in Numerical mathematics and advanced applications, Proceedings of ENUMATH 2009, G. Kreiss, P. L\"otstedt, A. Malqvist and M. Neytcheva eds., Lect. Notes Comput. Sci. Eng., Springer, 29--45 (2010).

\bibitem{clb_thomines} C. Le Bris and F. Thomines, {\em A Reduced Basis approach for some weakly stochastic multiscale problems}, Chinese Annals of Mathematics, 33B(5):657--672, 2012.

\bibitem{LM13} F. Legoll and W. Minvielle, {\em Variance reduction using antithetic variables for a nonlinear convex stochastic homogenization problem}, Discrete and Continuous Dynamical Systems - S, 8(1):1-27, 2015. 

\bibitem{M13} J.-C. Mourrat, {\em First-order expansion of homogenized coefficients under Bernoulli perturbations}, J. Math. Pures Appl., DOI 10.1016/j.matpur.2014.03.008, in press. 

\bibitem{nolen} J. Nolen, {\em Normal approximation for a random elliptic equation}, Probability Theory and Related Fields, DOI 10.1007/s00440-013-0517-9, published online in 2013.

\bibitem{papa} G.C. Papanicolaou and S.R.S. Varadhan, {\em Boundary value problems with rapidly oscillating random coefficients}, in Proc. Colloq. on Random Fields: Rigorous Results in Statistical Mechanics and Quantum Field Theory, J. Fritz, J.L. Lebaritz and D. Szasz, eds, Colloquia Mathematica Societ. Janos Bolyai, Vol. 10, North-Holland, Amsterdam, 1981, pp. 835--873.

\bibitem{shiryaev} A.N. Shiryaev, {\bf Probability},
Graduate Texts in Mathematics, vol. 95, Springer, 1984.

\bibitem{tempelman} A.A. Tempel'man, {\em Ergodic theorems for general
    dynamical systems}, Trudy Moskov. Mat. Obsc., 26:94--132, 1972.

\bibitem{yuri} V.V. Yurinski, {\em Averaging of symmetric diffusion in random medium}, Sibirskii Mat. Zh., 27(4):167--180, 1986.

\end{thebibliography}
\end{document}